\documentclass[11pt,reqno]{amsart}

\usepackage{amsmath,amsthm,amssymb,enumitem}

\newcommand{\Z}{\mathbb{Z}}

\newcommand{\R}{\mathbb{R}}

\DeclareMathOperator{\conv}{conv}

\DeclareMathOperator{\cl}{cl}

\newcommand{\innerProd}[2]{\langle#1,#2\rangle}
\newcommand{\norm}[2]{\|#2\|_{#1}}

\newcommand{\trace}{\ensuremath{\text{Tr}}}
\newcommand{\diag}{\ensuremath{\text{diag}}}

\newcommand{\secref}[1]{ \S\ref{#1}}

\newtheorem{proposition}{Proposition}
\newtheorem{theorem}{Theorem}
\newtheorem{corollary}{Corollary}
\newtheorem{lemma}{Lemma}
\theoremstyle{definition}

\newtheorem{example}{Example}
\theoremstyle{remark}
\newtheorem{remark}{Remark}
\usepackage{algorithm,algorithmicx}
\usepackage{algpseudocode}
\usepackage{bbm}
\usepackage[mathlines,pagewise,left]{lineno}
\usepackage{amssymb}
\usepackage{color,xcolor,tikz}
\usetikzlibrary{arrows,shapes,chains}  
\usepackage{amsmath}
\usepackage{bcol}
\usepackage{multirow, array}
\usepackage{caption}
\usepackage{xspace}
\usepackage[numbers]{natbib}
\usepackage{comment}
\usepackage[toc,page]{appendix}
\usepackage{graphicx} % Allows including images
\usepackage{booktabs} % Allows the use of \toprule, \midrule and \bottomrule in tables
\usepackage{pdflscape}
\usepackage[export]{adjustbox}
\usepackage[font=scriptsize]{subfig}
\usepackage{mathtools}
\usepackage{hyperref}

\usepackage{amsthm}
\theoremstyle{definition}
\theoremstyle{remark}

\newcommand{\be}[1]{\begin{equation}\label{#1}}
\newcommand{\ee}{\end{equation}}

\newcommand{\sdpfree}{\ensuremath{\textsf{OptRankOne}}\xspace}
\newcommand{\sdppos}{\ensuremath{\textsf{OptPairs}}\xspace}
\newcommand{\sdpstandard}{\ensuremath{\textsf{Shor}}\xspace}

\newcommand{\optPersp}{\ensuremath{\textsf{OptPersp}}\xspace}
\newcommand{\conicPersp}{\ensuremath{\textsf{ConicQuadPersp}}\xspace}
\newcommand{\decomp}{\ensuremath{\textsf{Decomp}}\xspace}

\newcommand{\conicN}{\textsf{ConicQuadN}\xspace}
\newcommand{\conicP}{\textsf{ConicQuadP}\xspace}
\newcommand{\conic}{\textsf{ConicQuadP+N}\xspace}
\newcommand{\perspS}{\optPersp}
\newcommand{\conicS}{\sdppos}
\newcommand{\nodes}{\texttt{Nodes}\xspace}
\newcommand{\rimp}{\texttt{Rimp}\xspace}
\newcommand{\Time}{\texttt{Time}\xspace}
\newcommand{\egap}{\texttt{Egap}\xspace}
\newcommand{\igap}{\texttt{igap}\xspace}

\newcommand{\CQI}{\ensuremath{\text{(QI)}}\xspace}
\newcommand{\I}[1]{\ensuremath{\mathcal{I}_{#1}}}
\newcommand{\D}{\ensuremath{\mathcal{D}}}
\newcommand{\W}{\ensuremath{\mathcal{W}}\space}
\newcommand{\set}[1]{\ensuremath{\mathcal{#1}}}
\newcommand{\symM}{\mathbb{S}}

\newcolumntype{\resetRow}{>{\global\let\currentrowstyle\relax}}
\newcolumntype{^}{>{\currentrowstyle}}
\newcommand{\rowstyle}[1]{\gdef\currentrowstyle{#1}%
	#1\ignorespaces
}

\def\SingleSpacedXI{\linespread{1.1}}
\SingleSpacedXI

\usepackage[colorinlistoftodos,prependcaption,textsize=small]{todonotes}
\title[${2\times 2}$-convexifications for MICQP]{$\mathbf{2\times 2}$-convexifications for convex quadratic optimization with indicator variables}
\author{Shaoning Han, Andr\'{e}s G\'{o}mez and Alper Atamt\"urk}

\thanks{ \noindent \hskip -5mm
	S. Han, A. G\'{o}mez : Daniel J. Epstein Department of Industrial and Systems Engineering, Viterbi School of Engineering, University of Southern California, CA 90089. \texttt{gomezand@usc.edu}, \texttt{shaoning@usc.edu}\\
A. Atamt\"urk: Department of Industrial Engineering \& Operations Research, University of California, Berkeley, CA 94720.
\texttt{atamturk@berkeley.edu}  
}
\begin{document}
	\maketitle
	
	\begin{abstract}
		\vskip 3mm
		\noindent 
In this paper, we study the convex quadratic optimization problem with indicator variables. For the bivariate case, we describe the convex hull of the epigraph in the original space of variables, and also give a conic quadratic extended formulation. Then, using the convex hull description for the bivariate case as a building block, we derive an extended SDP relaxation for the general case. This new formulation is stronger than other SDP relaxations
proposed in the literature for the problem, including Shor's SDP relaxation, the optimal perspective relaxation as well as the optimal rank-one relaxation.
Computational experiments indicate that the proposed formulations are quite effective in reducing the integrality gap of the optimization problems. \\
		 
\end{abstract}

\begin{center}
	April 2020 \\
\end{center}

\BCOLReport{20.02}

\section{Introduction}

\label{sec:intro}

We consider the convex quadratic optimization with indicators: 
%\todo{We do not use the linear terms in the objective in the study. We may omit them.}
\begin{equation}\label{eq:mixed_integer_problem}
\CQI \	\ \ \min\left\{ a'x+b'y+y'Qy: (x,y) \in \I{n} \right\},
\end{equation} 
where the indicator set is defined as
\[
\I{n} = \left \{(x,y)\in\{0,1\}^n\times\R_+^n: y_i(1-x_i)=0,\; \forall i\in[n] \right  \},
\]
where $a$ and $b$ are $n$-dimensional vectors, $Q\in\R^{n\times n}$ is a positive semidefinite (PSD) matrix and 
$[n] := \{1,2, \ldots, n\}$. For each $i \in [n]$, the complementary constraint $y_i (1-x_i) = 0$, along with the indicator variable $x_i \in \{0,1\}$, is used to state that $y_i=0$ whenever $x_i=0$. 
Numerous applications, including portfolio optimization \cite{bienstock1996computational}, optimal control  \cite{gao2011cardinality}, image segmentation  \cite{hochbaum2001efficient}, signal denoising \cite{bach2019submodular}
are either formulated as \CQI or can be relaxed to \CQI. %\eqref{eq:mixed_integer_problem}.

%In particular, \eqref{eq:mixed_integer_problem} encompasses the following problem as a special case by introducing indicator variables to incorporate the so-called $l_0$ norm
%\[ \min_w \norm{2}{y-Xw}^2+\lambda\norm{0}{w}, \]
%which is known as a NP-hard problem \citep{huo2010complexity}. Hence, problem \eqref{eq:mixed_integer_problem} is NP-hard in general. 

%To solve \eqref{eq:mixed_integer_problem} efficiently, a pivotal step is to construct strong convex relaxations. To obtain strong relaxations, many researchers study the structure of epigraph of $y'Qy$ to add valid inequalities for \eqref{eq:mixed_integer_problem}. Both 

Building strong convex relaxations of \CQI % \eqref{eq:mixed_integer_problem} 
is instrumental in solving it effectively. A number of approaches for developing linear and nonlinear valid inequalities for \CQI %\eqref{eq:mixed_integer_problem} 
are considered in literature. 
%\subsection*{Extended formulation via linearization}
 \citet{dong2013} describe lifted linear inequalities %for \eqref{eq:mixed_integer_problem}
  from its continuous quadratic optimization counterpart with bounded variables. \citet{bienstock2014cutting} derive valid linear inequalities for optimization of a convex objective function over a non-convex set  based on gradients of the objective function. Valid linear inequalities for \CQI %\eqref{eq:mixed_integer_problem} 
  can also be obtained using the epigraph of bilinear terms in the objective \cite[e.g.][]{ boland2017bounding,dey2019new, gupte2020extended,conv-bivariate}. In addition, several specialized results concerning optimization problems with indicator variables exist in the literature \cite{atamturk2020safe,belotti2016handling,bertsimas2019unified,bonami2015mathematical,dedieu2020learning,gomez2019outlier,gomez2020strong,lim2018valid,mahajan2017minotaur}.
 
 %\subsection*{Standard SDP reformulation} 
 A powerful approach for nonconvex quadratic optimization problems is semidefinite programming (SDP) reformulation,
 %This approach is also known as Shor semidefinite relaxation because it is
 first proposed by Shor \cite{shor1987quadratic}. % in 1987. 
 Specifically, a convex relaxation is constructed by introducing a rank-one matrix $Z$ representing $zz'$, where $z$ is the decision vector, and then forming the semidefinite relaxation $Z \succeq zz'$. Such SDP relaxations have been widely utilized in numerous applications, including max-cut problems \cite{goemans1995improved}, hidden partition problems of finding clusters in large network datasets \cite{javanmard2016phase},  matrix completion problems \cite{alfakih1999solving,candes2010matrix}, power systems \cite{FALA:conic-uc}, 
 robust optimization problems \cite{ben2009robust}.
 Sufficient conditions for exactness of SDP relaxations are also studied in the literature \cite[e.g.][]{burer2019exact,ho2017second,jeyakumar2014trust, wang2019tightness,wang2019generalized}.
 
%\subsection*{Perspective reformulation} 
There is a substantial body of research on the perspective formulation of convex univariate functions with indicators \cite{akturk2009strong,dong2015regularization,dong2013,frangioni2006perspective, gunluk2010perspective,
	hijazi2012mixed,wu2017quadratic}. When $Q$ is diagonal, $y'Qy$ is separable and the perspective formulation provides the convex hull of the epigraph of $y'Qy$ with indicator variables by strengthening each term $Q_{ii}y_i^2$ with its perspective counterpart $Q_{ii}y_i^2/x_i$, individually. For the general case, however, convex relaxations based on the perspective reformulation may not be strong. The computational experiments in \cite{frangioni2020decompositions} demonstrate that as $Q$ deviates from a diagonal matrix, the performance of the perspective formulation deteriorates.

%\subsection*{Convexification for two-variable terms}
Beyond the perspective reformulation, which is based on the convex hull of the epigraph of a univariate convex quadratic function with one indicator variable, the convexification for the bivariate case has received attention recently. Convex hulls of univariate and bivariate cases can be used as building blocks to strengthen \CQI % \eqref{eq:mixed_integer_problem} 
by decomposing $y'Qy$ into a sequence of low-dimensional terms.  
\citet{jeon2017quadratic} give conic quadratic valid inequalities for the bivariate case.
 \citet{frangioni2020decompositions} combine perspective reformulation and disjunctive programming and apply them to the bivariate case. %Then they minimize the norm of the residual term to get the best decomposition and based on it, they  construct a strong formulation for  \eqref{eq:mixed_integer_problem}. 
 \citet{atamturk2018strong} study the convex hull of the mixed-integer epigraph of $(y_1-y_2)^2$ with indicators.
\citet{atamturk2018signal} give the convex hull of the more general set
\[ \set Z_-:=\left\{ (x,y,t) \in \I{2} \times\R_+:t\ge d_1y_1^2-2y_1y_2+d_2y^2_2 \, \right \}, \]
with coefficients $d \in \set D := \{d \in \R^2: d_1\ge0, d_2\ge0, d_1d_2\ge1\}$. 
%\todo{can we use the smaller and common letter $a$ instead of $d$ for the parameters of the quadratic?\\ Shaoning: Since we use $d$ in our previous paper [5], and we need to compare the results about $Z_-$, we had better keep the consistency. I would interchange $Z_-$ and $Z_-$.}
The conditions on the coefficients $d_1, d_2$ imply convexity of the quadratic.
%hey provide the convex hull description of $Z_-$ and develop a convex relaxation for \eqref{eq:mixed_integer_problem}. Moreover, they adopt the cutting surface method to refine the convex approximation iteratively. Specifically, at each iteration, they find the optimal decomposition for the current solution by solving a separation problem. Then they add the valid inequalities associated with the decomposition to  strengthen the convex relaxation.
\citet{atamturk2019rank} study the case where the continuous variables are free and the rank of the coefficient matrix is one in the context of sparse linear regression. %\eqref{eq:mixed_integer_problem}. 
\citet{burer2020quadratic} give an extended SDP formulation for the convex hull of the $2 \times 2$ bounded set $\left\{(y,yy',xx'):0\le y\le x\in\{0,1 \}^2\right\}$. Their formulation does not assume convexity of the quadratic function and contain PSD matrix variables $X$ and $Y$ as proxies for $xx'$ and $yy'$ as additional variables. 
\citet{lowdim-quad} study computable representations of convex hulls of low dimensional quadratic forms without indicator variables.

%\subsection*{Decomposition and generalization }
%In order to obtain a strong convex formulation for \eqref{eq:mixed_integer_problem} by utilizing the results of low-dimensional cases, a natural idea is decomposing $y'Qy$ into low-dimensional terms and a PSD residual term and then strengthening each low-dimensional term by its convex envelope. \citet{frangioni2019decompositions} follow this direction to get the best decomposition by minimizing the norm of the residual term. 

%more study based on convexification with two indicator variables have been carried out by researchers recently. Once the low-dimensional cases are understood well, a strong convex formulation for \eqref{eq:mixed_integer_problem} can be obtained by decomposing $y'Qy$ into a sequence of low-dimensional terms and a PSD residual term and then strengthening each low-dimensional term by its convex envelope. 

\subsection*{Contributions} There are three main contributions in this paper.
\subsubsection*{1. We show the equivalence between the optimal perspective reformulation and Shor's SDP formulation for \CQI} %\eqref{eq:mixed_integer_problem}}

These two formulations have been studied extensively in literature. While it is known that Shor's SDP formulation is at least as strong as the perspective formulation \cite{dong2015regularization}, the other direction has not been explored. We show in this paper that these two formulations are in fact equivalent.

\subsubsection*{2. Bivariate case: We describe the convex hull of the epigraph of a convex bivariate quadratic with a positive cross product and indicators.}

Consider  
\[ \set Z_+:=\left\{ (x,y,t)\in \I{2}\times\R_+:t\ge d_1y_1^2+2y_1y_2+d_2y^2_2 \, \right\}, \]
where $d \in \set D$. Observe that any bivariate convex quadratic with positive off-diagonals can be written as $d_1y_1^2+2y_1y_2+d_2y^2_2$ by scaling appropriately. Therefore, $\set Z_+$ is the \textit{complementary} set to $\set Z_-$ and, together, $\set Z_+$ and $\set Z_-$
model epigraphs of all bivariate convex quadratics with indicators.

 In this paper, we propose \textit{conic quadratic} extended formulations to describe $\cl\conv(\set Z_-)$ and $\cl\conv(\set Z_+)$.% respectively using techniques from second order cone programming. 
 %The new formulations can be reformulated as a part of an SDP relaxation of problem \CQI. %\eqref{eq:mixed_integer_problem}. 
 We also give the explicit description of $\cl\conv(\set Z_+)$ in the original space of the variables. The corresponding convex envelope of the bivariate function is a four-piece function. While the ideal formulations of $\set Z_-$ can be conveniently described with two simpler valid ``extremal" inequalities \cite{atamturk2018signal}, a similar result does not hold for $\set Z_+$ (see Example~\ref{ex:difference} in \S\ref{sec:convexHull}). The derivation of ideal formulations for the more involved set $\set Z_+$ differs significantly from the methods in \cite{atamturk2018signal}. The complementary results of this paper and \cite{atamturk2018signal} for $\set Z_-$ complete the convex hull descriptions of bivariate convex functions with indicators.

\subsubsection*{3. General case: We develop an optimal SDP relaxation based on $2\times2$ convexifications for \CQI} %\eqref{eq:mixed_integer_problem}}
 In order to construct a strong convex formulation for \CQI, %\eqref{eq:mixed_integer_problem}, 
 we  extract a sequence of  $2\times 2$ PSD matrices from $Q$ such that the residual term is a PSD matrix as well, and convexify each bivariate quadratic term utilizing the descriptions of $\cl\conv(\set Z_+)$ and $\cl\conv(\set Z_-)$. This approach works very well when $Q$ is $2\times2$ PSD decomposable, i.e., when $Q$ is scaled-diagonally dominant 
 \cite{boman2005factor}.
 Otherwise, a natural question is how to optimally decompose $y'Qy$ into bivariable convex quadratics and a residual convex quadratic term so as to achieve the best strengthening. 
 
 We address this question by deriving an optimal convex formulation using SDP duality.
 The new SDP formulation dominates any formulation obtained through a $2\times2$-decomposition scheme. This formulation is also stronger than other SDP formulations in the literature (see Figure~\ref{fig:formulationRelation}), including the optimal perspective formulation \cite{dong2015regularization}, Shor's SDP formulation \cite{shor1987quadratic}, and the optimal rank-one convexification \cite{atamturk2019rank}. In addition, the proposed formulation is solved many orders of magnitude faster than the $2\times2$-decomposition approaches based on disjunctive programming \cite{frangioni2020decompositions}.
 
% Also, \citet{atamturk2019rank} propose an SDP formulation to convexify mixed-integer quadratic optimization problems with free continuous variables based on the convex hull of a rank-one function. However, when applied to \eqref{eq:mixed_integer_problem}, since the formulation proposed by \cite{atamturk2019rank} neither exploits the higher-rank structure nor accounts for the  nonnegativity of the continuous variables, we can show the new formulation is stronger than it.

\begin{figure}[h]
	\scriptsize
	\tikzstyle{existingFormulation} = [rectangle,rounded corners, draw, thick, text width = 2cm,  fill=white, text centered]
	\tikzstyle{novelFormulation} = [ellipse, text width = 2cm, draw, thick, fill=white,text centered]
	\tikzstyle{point} = [coordinate, on grid]
	\tikzstyle{arrow} = [thick,->,>=latex,double]
	\tikzstyle{arrow2} = [thick,->,>=latex,double,dashed]
	\begin{tikzpicture}[node distance = 1.5cm]
	\node[existingFormulation](basic){Natural QP relaxation};		
	\node[point, right of=basic, xshift = 1.5cm](midpt){};
	\node[existingFormulation,above of = midpt](persp){Optimal perspective \optPersp};
	\node[existingFormulation,below of=midpt](sdps){\sdpstandard SDP};
	\node[existingFormulation, right of= midpt,xshift = 1.5cm](sdpf){Optimal rank-one \sdpfree};
	\node[novelFormulation,right of=sdpf,xshift = 2cm](sdpp){Optimal pairs \sdppos};
	\draw[arrow](basic)--(persp);
	\draw[arrow](basic)--(sdps);
	\draw[arrow](persp) to [bend right] (sdps);
	\draw[arrow2](sdps) to [bend right] (persp);
	\draw[arrow](persp)--(sdpf);
	\draw[arrow2](sdps)--(sdpf);
	\draw[arrow2](sdpf)--(sdpp);
	\end{tikzpicture}
	\caption{\footnotesize Relationship between the convex relaxations for \CQI discussed in this paper. Rectangular frames and circle frames indicate formulations in the literature and the new formulation in this paper, respectively. The arrow direction \textsf{A}$\to$\textsf{B} indicates that formulation \textsf{B} is stronger than formulation \textsf{A}. Solid and dashed lines indicate existing relations in the literature and relations shown in this paper, respectively.}\label{fig:formulationRelation}
\end{figure}
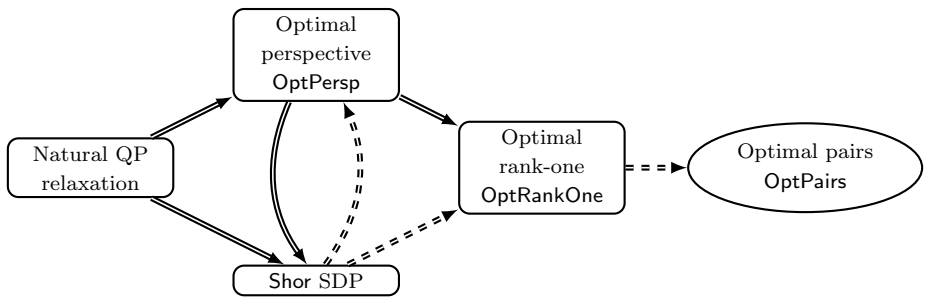

%\todo{A:Not clear what Basic/QP relaxation is. Are complementary constraints dropped in addition to relaxing binary $x$ to $0\le x \le 1$?}

\subsection*{Outline} The rest of the paper is organized as follows. In \secref{sec:previousResults} we  review the optimal perspective formulation and Shor's SDP formulation for \CQI  
and show that these two formulations are equivalent. In \secref{sec:convexHull} we provide a conic quadratic formulation of $\cl\conv(\set Z_+)$ and $\cl\conv(\set Z_-)$ in an extended space and derive the explicit form of $\cl\conv(\set Z_+)$ in the original space. In \secref{sec:formulation}, employing the results in \secref{sec:convexHull},  we give a strong convex relaxation for \CQI % \eqref{eq:mixed_integer_problem} 
using SDP techniques.  In \secref{sec:comparison}, we compare the strength of the proposed SDP relaxation with others in literature. In \secref{sec:computation}, we present computational results demonstrating the effectiveness of the proposed convex relaxations. 
Finally, in \secref{sec:conclusions}, we conclude with a few final remarks.

\subsection*{Notation} Throughout, we adopt the following convention for division by $0$: given $a\geq 0$, $a^2/0=\infty$ if $a\neq 0$ and $a^2/0=0$ if $a=0$. For a set $\set{X}\subseteq \R^n$, $\cl\conv(\set X)$ denotes the closure of the convex hull of $\set X$. 
For a vector $v$, $\diag(v)$ denotes the diagonal matrix $V$ with $V_{ii} = v_i$ for each $i$. 
%For a square matrix $V$, $\diag(V)$ denotes the vector $v$ with $v_i =V_{ii}$'s 
%\todo{Not sure I understand this notation. Alper: dropped matrix to vector as it is not used.}
Finally, $\symM_+^n$ refers to the cone of $n\times n$ real PSD matrices.

\section{Optimal perspective formulation vs. Shor's SDP}\label{sec:previousResults}
In this section we analyze two well-known convex formulations: the optimal perspective formulation and Shor's SDP. We first introduce the two formulations and then show that when applied to \CQI, they are equivalent.
%\todo{Is the adjective "standard" used for this formulation in the literature?\\
%Shaoning: I guess Shor's SDP is better. At least I saw this terminology in literature.}

First, we consider set
\[\set X_0 = \left\{ (x,y,t)\in\{0,1 \}\times\R^2_+:t\ge y^2,(1-x)y=0 \right\}. \]
The convex hull of $\set X_0$ can be described with the perspective reformulation
\[ \cl\conv(\set X_0)=\left\{ (x,y,t)\in[0,1 ]\times\R^2_+:t\ge y^2/x \right\}. \] 
Splitting $Q$ into some diagonal PSD matrix $D$ and a PSD residual, i.e., $Q-D\succeq 0$, one can apply the perspective reformulation to each diagonal term, by replacing $D_{ii}y_i^2$ with $D_{ii}y_i^2/x_i$,
to get a valid convex relaxation of \CQI --after relaxing integrality constraints in $x$ and dropping the complementary constraints $y_i(1-x_i)=0$. %eqref{eq:mixed_integer_problem}.
\citet{dong2015regularization} describe an optimal perspective relaxation for \CQI. %eqref{eq:mixed_integer_problem}. 
They show that 
every such perspective relaxation is dominated by the optimal perspective relaxation: % which is denoted by \optPersp.
\begin{subequations}\label{eq:perspective}
%\begin{equation}\label{eq:perspective}\tag{\optPersp}
\begin{align}
\min\;& a'x+b'y+\innerProd{Q}{Y}\\
(\optPersp)\qquad\qquad\text{s.t.}\;&Y-yy'\succeq 0\label{eq:perspective_psd}\\
&y_i^2\leq Y_{ii}x_i&\forall i\in[n]\label{eq:perspective_rotated}\\
&0\le x \le 1,\;y\ge0.
\end{align}
%\end{equation}
\end{subequations}

Next, we consider Shor's SDP relaxation for problem~\CQI: %eqref{eq:mixed_integer_problem}.
%\begin{equation}\tag{\sdpstandard}
\begin{subequations}
\begin{align}
\min\;& a'x+b'y+\sum_{i=1}^n\sum_{j=1}^n Q_{ij}Z_{ij}\\
\text{s.t.}\;&y_i-Z_{i,i+n}=0 & \forall i\in[n]\\
(\sdpstandard)\qquad \qquad &x_i-Z_{i+n,i+n}=0&\forall i\in[n]\\
&Z-\begin{pmatrix}
y\\x
\end{pmatrix}\begin{pmatrix}
y'&x'
\end{pmatrix}\succeq 0\label{eq:sdp_psd}\\
&0\le x \le 1,\;y\ge0,
\end{align}
\end{subequations}
where $Z\in \R^{2n\times 2n}$ such that  $Z_{ii}$ is a proxy for $y_i^2$,  $Z_{i+n,i+n}$ is a proxy for $x_i^2$, $Z_{i,i+n}$ is a proxy for $x_iy_i$, $i\in[n]$, and $Z_{ij}$ is a proxy for $y_iy_j$ for $1\leq i,j\leq n$. 

It is known that \sdpstandard\ is at least as strong as \optPersp \cite{dong2013}, as constraints \eqref{eq:perspective_rotated} are implied from the positive definiteness of some $2\times 2$ principal minors of \eqref{eq:sdp_psd}. We show below that the two formulations are, in fact, equivalent. As \optPersp is a much smaller formulation than \sdpstandard, it is preferred.
\begin{theorem}\label{thm:equivalence}
	\sdpstandard\ is equivalent to \optPersp.
\end{theorem}
\begin{proof}
	First we verify \sdpstandard\ is at least as strong as \optPersp by checking that constraints \eqref{eq:perspective_psd}--\eqref{eq:perspective_rotated} are implied from \eqref{eq:sdp_psd}.  
	%Assume $(x,y,Z)$ is a feasible solution to \sdpstandard. We need to show that there exists a feasible solution to \optPersp with the same objective value.
	Let $Y_{ij} = Z_{ij}$ for any $1\le i,j\le n$. By Schur Complement Lemma,
	\[ Z-\begin{pmatrix}
	y\\x
	\end{pmatrix}\begin{pmatrix}
	y'&x'
	\end{pmatrix}\succeq 0\iff \begin{pmatrix}
	1&\begin{matrix}
	y'&x'
	\end{matrix}\\
	\begin{matrix}
	y\\x
	\end{matrix}&Z
	\end{pmatrix}\succeq 0. \]
	Since $Y$ is a principle submatrix of $Z$, we have \[ \begin{pmatrix}
	1&y'\\y&Y
	\end{pmatrix} \succeq0 \Leftrightarrow Y-yy'\succeq 0.\] Moreover, constraint \eqref{eq:sdp_psd} also implies that for any $i=1,2,\dots,n$,
	\begin{equation}\label{eq:psdSubmatrices}
		\begin{pmatrix}
		Z_{ii}&Z_{i,i+n}\\Z_{i,i+n}&Z_{i+n,i+n}
		\end{pmatrix}\succeq 0.
	\end{equation} 
	After substituting $Y_{ii} = Z_{ii}, \ x_i =  Z_{i+n,i+n},$ and $y_i = Z_{i,i+n}$, we find that \eqref{eq:psdSubmatrices} is equivalent to $Y_{ii}x_i\ge y_i^2$ in \optPersp, concluding the argument. 
	
	We next prove that \optPersp is at least as strong as \sdpstandard, by showing that for any given feasible solution $(x,y,Y)$ of  \optPersp it is possible to construct a feasible solution of \sdpstandard\ with same objective value. We rewrite $Z$ in the form  $Z=\begin{pmatrix}
	Y&U\\U'&V
	\end{pmatrix}.$ 
	For a fixed feasible solution $(x,y,Y)$ of \optPersp consider the optimization problem
	\begin{subequations}\label{eq:auxPrimal}
		\begin{align}
		\lambda^*:=&\min_{\lambda,U,V}\;\lambda\\
		\text{s.t. }&\begin{pmatrix}
		1&y'&x'\\
		y&Y &U\\
		x&U'&V
		\end{pmatrix}+\lambda I\succeq 0 \label{eq:primal_matrix_inequality}\\
		&U_{ii}=y_i,&\forall i\in[n] \,\, \\
		&V_{ii}=x_i,&\forall i\in[n],
		\end{align}
	\end{subequations}
	where $I$ is the  identity matrix. Observe that if $\lambda^*\le0$, then an optimal solution of \eqref{eq:auxPrimal} satisfies \eqref{eq:sdp_psd} and thus induces a feasible solution of \sdpstandard. We show next that this is, indeed, the case.
	
    One can verify that the strong duality holds for \eqref{eq:auxPrimal} since $\lambda$ can be an arbitrary positive number to make the matrix inequality hold strictly. Let $\tilde{Y}=\begin{pmatrix}
	1&y'\\y&Y
	\end{pmatrix}$ and consider the SDP dual of \eqref{eq:auxPrimal}:
	\begin{subequations}
		\begin{align} \label{eq:auxDual}
		\lambda^*=\max_{R,s,t,z} \;&-\innerProd{\tilde{Y}}{R} - \sum_i( 2x_iz_i+t_ix_i+2s_iy_i) &\\
		\text{s.t. }& \begin{pmatrix}
		R&\begin{matrix}
		z'\\\diag(s)
		\end{matrix}\\\label{eq:dual_variable}
		\begin{matrix}
		z&\diag(s)
		\end{matrix}&\diag(t)
		\end{pmatrix}\succeq 0\\
		&\trace(R) + \sum_i t_i =1,&\label{eq:normalization}
		\end{align}
	\end{subequations}
where $R, z, \diag(s), \diag(t)$ are the dual variable associated with $\tilde{Y}+\lambda I, x, U,$ and $V+\lambda I$, respectively. Note that we abuse the symbol $I$ to represent the identity matrices of different dimensions. Because the off-diagonal elements of $U$ and $V$ do not appear in the primal objective function and constraints  other than \eqref{eq:primal_matrix_inequality}, the corresponding dual variables are zero.

	Note that to show $\lambda^*\le0$,  it is sufficient to consider a relaxation of \eqref{eq:auxDual}. Therefore, dropping
	\eqref{eq:normalization}, 
	%is a normalization constraint which can be dropped. 	\todo{$z$ is not in the normalization constraint, but in the dual objective. Doesn't that affect the sign of $\lambda^*$? \\ Shaoning: No, we can consider it in two ways. 1. the coefficients before $z_i$ in this constraint is 0; 2. If we drop this constraint, the whole problem is homogeneous, so this constraint won't effect the sign of $\lambda^*$.}
	it is sufficient to show that
	\begin{equation}
 \innerProd{\tilde{Y}}{R} + \sum_i( 2x_iz_i+t_ix_i+2s_iy_i)\geq 0\label{eq:objective}
	\end{equation} for all $t\ge0$, $s$, $z,$ and $R$ satisfying \eqref{eq:dual_variable}. 

	Observe that if $t_i=0$, then $s_i=z_i=0$ in any solution satisfying \eqref{eq:dual_variable}. In this case, all such terms indexed by $i$ vanish in \eqref{eq:objective}. Therefore, it suffices to prove $\eqref{eq:objective}$ holds for all $t>0$.
	
%	It suffices to prove $\eqref{eq:objective}$ for all $t>0$.  If the conclusion holds for all $t$ s.t. $t>0$, then for any $t\ge 0$, we can consider $t+\epsilon e$, where $e$ is a vector with all elements as $1$'s, and we have $\innerProd{\tilde{Y}}{R} + \sum_i( 2x_iz_i+(t_i+\epsilon)x_i+2s_iy_i)\ge0$, i.e. $\innerProd{\tilde{Y}}{R} + \sum_i( 2x_iz_i+t_ix_i+2s_iy_i)\ge -\epsilon(\sum_i x_i)\ge -n \epsilon$. Since $\epsilon$ can be arbitrary small, we get $\innerProd{\tilde{Y}}{R} + \sum_i( 2x_iz_i+t_ix_i+2s_iy_i)\ge0$.
	
	For $t>0$, by Schur Complement Lemma, \eqref{eq:dual_variable} is equivalent to
	\begin{equation}
	R\succeq \begin{pmatrix}
	z'\\\diag(s)
	\end{pmatrix}\diag^{-1}(t)\begin{pmatrix}
	z&\diag(s)
	\end{pmatrix}.\label{eq:dual_variable_equiv}
	\end{equation}
	Moreover, since $\tilde{Y}\succeq0$, we find that $$\innerProd{\tilde{Y}}{R}\ge\innerProd{\tilde{Y}}{\begin{pmatrix}
		z'\\\diag(s)
		\end{pmatrix}\diag^{-1}(t)\begin{pmatrix}
		z&\diag(s)
		\end{pmatrix}}$$
	whenever $R$ satisfies \eqref{eq:dual_variable_equiv}. Substituting the term $\innerProd{\tilde{Y}}{R}$ in \eqref{eq:objective} by its lower bound, it suffices to show that 	
	\begin{equation}\label{eq:lb_objecive}
		\innerProd{\tilde{Y}}{\begin{pmatrix}
			z'\\\diag(s)
			\end{pmatrix}\diag^{-1}(t)\begin{pmatrix}
			z&\diag(s)
			\end{pmatrix}} + \sum_i( 2x_iz_i+t_ix_i+2s_iy_i)\ge0,
	\end{equation} 
	holds for all $t>0,s, z$ and $R$ satisfying \eqref{eq:dual_variable_equiv}. A direct computation shows that
	\[\begin{pmatrix}
	z'\\\diag(s)
	\end{pmatrix}\diag^{-1}(t)\begin{pmatrix}
	z&\diag(s)
	\end{pmatrix}=\begin{pmatrix}
	\sum_i z_i^2/t_i&z_1s_1/t_1&\cdots&z_ns_n/t_n\\
	z_1s_1/t_1&s_1^2/t_1&&\\
	\vdots&&\ddots&\\
	z_ns_n/t_n&&&s_n^2/t_n
	\end{pmatrix}\]
	with all off-diagonal elements equal to $0$, except for the first row/column.
%	\[\begin{pmatrix}
%	z'\\\diag(s)
%	\end{pmatrix}\diag^{-1}(t)\begin{pmatrix}
%	z&\diag(s)
%	\end{pmatrix}=\begin{pmatrix}
%	\sum_i \frac{z_i^2}{t_i}&\frac{z_1s_1}{t_1}&\frac{z_2s_2}{t_2}&\cdots&\frac{z_ns_n}{t_n}\\
%	\frac{z_1s_1}{t_1}&\frac{s_1^2}{t_1}&&&\\
%	\frac{z_2s_2}{t_2}&&\frac{s_2^2}{t_2}&&\\
%	\vdots&&&\ddots&\\
%	\frac{z_ns_n}{t_n}&&&&\frac{s_n^2}{t_n}
%	\end{pmatrix}.\]
	 Thus, \eqref{eq:lb_objecive} reduces to the separable expression
\[ \sum_i \left( \frac{z_i^2}{t_i}+\frac{2z_is_iy_i}{t_i}+\frac{s_i^2}{t_i}Y_{ii}+2x_iz_i+x_it_i+2y_is_i \right)\ge0.  \]
	 For each term, we have
	\begin{align*}
	&\frac{z_i^2}{t_i}+\frac{2z_is_iy_i}{t_i}+\frac{s_i^2}{t_i}Y_{ii}+2x_iz_i+x_it_i+2y_is_i\\
	= \ &(z_i^2+2z_is_iy_i+s_i^2Y_{ii})/t_i+x_it_i+2x_iz_i+2y_is_i\\
	\ge \ & 2\sqrt{x_i(z_i^2+2z_is_iy_i+s_i^2Y_{ii})}+2x_iz_i+2y_is_i\\
	\ge \ &0,
	\end{align*}
	where the first inequality follows from the mean-value inequality $a+b\ge 2\sqrt{ab}$ for $a,b\ge0$.
	The last inequality holds trivially if $2x_iz_i+2y_is_i\geq 0$;  otherwise, we have
	\begin{align*}
	&\sqrt{x_i(z_i^2+2z_is_iy_i+s_i^2Y_{ii})}\ge -(x_iz_i+y_is_i)\\
	\iff&x_i(z_i^2+2z_is_iy_i+s_i^2Y_{ii})\ge(x_iz_i+y_is_i)^2\\
	\iff&x_iz_i^2(1-x_i)+s_i^2(x_iY_{ii}-y_i^2)\ge 0. \tag{as $0\le x_i\le 1$ and $x_iY_{ii}\ge y_i^2$}
	\end{align*}
	In conclusion, $\lambda^*\le0$ and this completes the proof.
\end{proof}
  %On the one hand, since \sdpstandard introduces more variables than \optPersp,  the solution delivered by \sdpstandard is as good as, if not better than, the solution delivered by \optPersp. On the other hand, as we only specify a small number of elements of $Z$, the SDP constraint in \sdpstandard does not provide more restrictions than that the submatrices of form \eqref{eq:psdSubmatrices} should be PSD. This results in the equivalence of \sdpstandard and \optPersp.
  
%% end of section 2

\section{Convex hull description of $\set Z_+$}\label{sec:convexHull}
In this section, we give ideal convex formulations for
\[ 
\set Z_+=\left\{ (x,y,t)\in \I{1}\times\R_+:t\ge d_1y_1^2+2y_1y_2+d_2y_2^2 \,  \right\}. 
\]
 When $d_1=d_2=1$, $\set Z_+$ reduces to the simpler rank-one set % with two nonnegative continuous variables
\[ 
\set X_+=\left\{ (x,y,t)\in\I{1}\times\R_+:t\ge (y_1+y_2)^2 \, \right\}. 
\] 
Set $\set X_+$ is of special interest as it arises naturally in \CQI when $Q$ is a diagonal dominant matrix, see computations in \S\ref{sec:comp_synt} for details.
As we shall see, the convex hulls of $\set Z_+$ and $\set X_+$ are significantly more complicated than their 
complementary sets $\set Z_-$ and $\set X_-$ studied earlier.
 In \secref{sec:extendedFormulation},  we develop an SOCP-representable extended formulation of $\cl\conv(\set Z_+)$. Then, in \secref{sec:convexHullDescription}, we derive the explicit form of $\cl\conv(\set Z_+)$ in the original space of variables.

\subsection{Conic quadratic-representable extended formulation} \label{sec:extendedFormulation} We start by 
writing $\set Z_+$ as the disjunction of four convex sets defined by all values of the indicator variables; that is,
 \[\set Z_+= \set Z_+^1\cup \set Z_+^2\cup \set Z_+^3\cup \set Z_+^4,\] where $\set Z_+^i, i=1,2,3,4$ are convex sets defined as:
 \begin{align*}
 	&\set Z_+^1=\{ (1,0, u, 0,t_1):t_1\ge d_1 u^2, u\ge0 \},\\
 	&\set Z_+^2=\{ (0,1, 0, v,t_2):t_2\ge d_2 v^2, v\ge0 \},\\
 	&\set Z_+^3=\{ (1,1,w_1,w_2,t_3):t_3\ge d_1 w_1^2 + 2 w_1w_2+ d_2 w_2^2, w_1\ge0, w_2\ge0  \},\\
 	&\set Z_+^4 = \{ (0,0,0,0,t_4):t_4\ge0 \}.
 \end{align*}
  By the definition, a point $(x_1,x_2,y_1,y_2,t)\in\conv(\set Z_+)$ if and only if it can be written as a convex combination of four points belonging in $\set Z_+^i, i=1,2,3,4$. Using $\lambda=(\lambda_1,\lambda_2,\lambda_3,\lambda_4)$ as the corresponding weights, $(x_1,x_2,y_1,y_2,t)\in\conv(\set Z_+)$ if and only if the following inequality system has a feasible solution
 \begin{subequations} \label{eq:disjunctZplus}	
 	\begin{align}
 	&\lambda_1+\lambda_2+\lambda_3+\lambda_4=1 \label{eq:disjunctZplus1}\\
 	&x_1=\lambda_1+\lambda_3, \; x_2=\lambda_2+\lambda_3 \label{eq:disjunctZplus2}\\
 	&y_1=\lambda_1u+\lambda_3w_1,\; y_2=\lambda_2v+\lambda_3  w_2\label{eq:disjunctZplus3}\\
 	&t=\lambda_1t_1+\lambda_2t_2+\lambda_3t_3+\lambda_4t_4\label{eq:disjunctZplus4}\\
 	&t_1\ge d_1u^2,\; t_2\ge d_2v^2,\; t_3\ge d_1w_1^2+2w_1w_2+d_2w_2^2,\; t_4\ge 0\label{eq:disjunctZplus5}\\
 	&u,v,w_1,w_2,\lambda_1,\lambda_2,\lambda_3,\lambda_4\ge0.
 	\end{align}
 \end{subequations}
%Note that a convex extended formulation of $\cl\conv(Z_+)$ could also be obtained using the technique in disjunctive programming \cite[see][]{ceria1999convex}. Also, \citet{vielma2019small} demonstrate how to call Cayley embedding to eliminate the auxiliary variables.

 We will now simplify \eqref{eq:disjunctZplus}. First, by Fourier–Motzkin elimination, one can substitute $t_1,t_2,t_3,t_4$ with their lower bounds in \eqref{eq:disjunctZplus5} and reduce \eqref{eq:disjunctZplus4} to $t\ge \lambda_1d_1u^2+\lambda_2d_2v^2+\lambda_3(d_1w_1^2+2w_1w_2+d_2w_2^2)$. Similarly, since $\lambda_4\ge0$, one can eliminate $\lambda_4$ and reduce \eqref{eq:disjunctZplus1} to $\sum_{i=1}^3 \lambda_i\le 1$.   Next, using \eqref{eq:disjunctZplus3}, one can substitute $u=(y_1-\lambda_3w_1)/\lambda_1$ and $v=(y_2-\lambda_3w_2)/\lambda_2$. Finally, using \eqref{eq:disjunctZplus2}, one can substitute $\lambda_1=x_1-\lambda_3$ and $\lambda_2=x_2-\lambda_3$ to arrive at 
 % For the sake of convenience, let $\lambda$ denote $\lambda_3$ and the inequality system \eqref{eq:disjunctZplus} reduces to
\begin{subequations}\label{eq:posDisjunct}	
	\begin{align}
	&	\max\{ 0,x_1+x_2-1 \}\le\lambda_3\le\min\{ x_1,x_2 \} \label{eq:posDisjunctA}\\
	& \lambda_3 w_i\le y_i,\;\; i=1,2\label{eq:posUpper}\\
	&w_i\ge0,\;\; i=1,2\label{eq:posDisjunctC}\\
	&t\ge \frac{d_1 (y_1 \!- \!\lambda_3 w_1)^2}{x_1 \! - \!\lambda_3} \! + \! 
	\frac{d_2(y_2 \! - \! \lambda_3 w_2)^2}{x_2 \! - \! \lambda_3} \! + \!
	\lambda_3(d_1w_1^2+2w_1w_2+d_2w_2^2)\label{eq:posDisjunctD},
	\end{align}
\end{subequations}
where \eqref{eq:posDisjunctA} results from the nonnegativity of $\lambda_1,\lambda_2,\lambda_3,\lambda_4$, \eqref{eq:posUpper} from the nonnegativity of $u$ and $v$.  Finally,	observe that \eqref{eq:posUpper} is redundant for \eqref{eq:posDisjunct}: indeed, if there is a solution $(\lambda, w, t)$ satisfying \eqref{eq:posDisjunctA}, \eqref{eq:posDisjunctC} and \eqref{eq:posDisjunctD} but violating \eqref{eq:posUpper}, one can decrease $w_1$ and $w_2$ such that \eqref{eq:posUpper} is satisfied without violating \eqref{eq:posDisjunctD}.

Redefining variables in \eqref{eq:posDisjunct}, we arrive at the following conic quadratic-representable extended formulation for $\cl\conv(\set Z_+)$ 
and its rank-one special case $\cl\conv(\set X_+)$.
\begin{proposition}\label{prop:extendedPos}
	The set $\cl\conv(\set Z_+)$ can be represented as
	\begin{align*}
	\cl\conv(\set Z_+)=\bigg\{& (x,y,t)\in[0,1]^2\times \R_+^3: \exists \lambda\in\R_+,z\in\R_+^2 
	\text{ s.t.} \\ 
	&x_1+x_2-1 \le\lambda\le\min\{ x_1,x_2 \}, \\
	&\begin{matrix}
	t\ge \frac{d_1(y_1-z_1)^2}{x_1-\lambda}+\frac{d_2(y_2-z_2)^2}{x_2-\lambda}
	+\frac{d_1z_1^2+2z_1z_2+d_2z_2^2}{\lambda} \bigg\} \cdot
	\end{matrix}
	\end{align*}
\end{proposition}
%\begin{proof}
%The conclusion is drawn by letting $z_i=\lambda w_i,i=1,2$.
%\end{proof}
\begin{corollary}
	The set $\cl\conv(\set X_+)$ can be represented as
	\begin{align*}
	\cl\conv(\set X_+)=\bigg\{& (x,y,t)\in[0,1]^2\times \R_+^3: \exists \lambda\in\R_+,z\in\R_+^2  
	\text{ s.t. }\\ &x_1+x_2-1 \le\lambda\le\min\{ x_1,x_2 \}, \\ 
	&t\ge \frac{(y_1-z_1)^2}{x_1-\lambda}+\frac{(y_2-z_2)^2}{x_2-\lambda}+\frac{(z_1+z_2)^2}{\lambda} &\bigg\} \cdot
	\end{align*}
\end{corollary}

\begin{remark}\label{remark:extendedNeg} 
	%\todo{This remark is long\\ Shaoning: should we compare the two extended formulations in details? This is not the focus of our paper. If we want to explain it clearly, we need at least one page and had better make this part a subsection itself.}
% Defined in intro section already
%Similarly, we can consider $Z_-$ and $X_-,$ where
%\[ Z_-=\left\{ (x,y,t)\in\I{1}\times\R_+:t\ge d_1y_1^2-2y_1y_2+d_2y_2^2 \, \right\}, \]
%and 
%\[ X_-=\left\{ (x,y,t)\in\I{1} \times\R_+:t\ge (y_1-y_2)^2 \right\}. \] 
One can apply similar arguments to the complementary set $\set Z_-$ %and its rank-one special case $X_-$ 
to derive an SOCP representable formulation of its convex hull as
	\begin{align*}
	\cl\conv(\set Z_-)=\bigg\{& (x,y,t)\in[0,1]^2\times \R_+^3: \exists \lambda\in\R_+,z\in\R^2 \text{ s.t. }\\ 
	&x_1+x_2-1 \le\lambda\le\min\{ x_1,x_2 \}, \\ &	 z_1\le y_1, \ z_2\le y_2, \\
	&\begin{matrix}
	t\ge \frac{d_1(y_1-z_1)^2}{x_1-\lambda}+\frac{d_2(y_2-z_2)^2}{x_2-\lambda}+\frac{d_1z_1^2-2z_1z_2+d_2z_2^2}{\lambda} \bigg\} \cdot
	\end{matrix}
	\end{align*}
This extended formulation is smaller than the one given \citet{atamturk2018signal} for $\cl\conv(\set Z_-)$.
\ignore{	and %a new extended formulation of $\cl\conv(X_-)$ as
	\begin{align*}
	\cl\conv(X_-)=\bigg\{& (x,y,t)\in[0,1]^2\times \R_+^3: \exists \lambda\in\R_+,z\in\R^2 
	\text{ s.t. } \\ &x_1+x_2-1 \le\lambda\le\min\{ x_1,x_2 \}, \\ & z_1\le y_1, \ z_2\le y_2, \\ 
	&t\ge \frac{(y_1-z_1)^2}{x_1-\lambda}+\frac{(y_2-z_2)^2}{x_2-\lambda}+\frac{(z_1-z_2)^2}{\lambda} &\bigg\} \cdot
	\end{align*}
}

	\ignore{
	\begin{align*}
	\cl\conv(Z_-)= \bigg\{&
	(x,y,t)\in[0,1]^2\times\R_+^3: \exists t_1,t_2,q_1,q_2\in\R_+ \text{ s.t. } \\
	&	y_1^2\le t_1x_1,\; y_2^2\le t_2x_2\\
	&	d_1v_1\ge d_1y_1-y_2,\;v_1^2\le q_1x_1\\
	&	d_1v_2\ge-d_1y_1+y_2,\;v_2^2\le 1_1x_2\\
	&	d_1q_1+t_2\left(d_2-\frac{1}{d_1}\right)\le t\\
	&	d_2w_1\ge y_1-d_2y_2,\;w_1^2\le q_2x_1\\
	&	d_2w_2\ge-y_1+d_2y_2,\;w_2^2\le q_2x_2\\
	&	d_2q_2+t_1\left(d_1-\frac{1}{d_2} \right)\le t \bigg \}
\end{align*}

	\begin{align*}
	\cl\conv(X_-)=\bigg\{& (x,y,t)\in[0,1]^2\times\R_+^3:\exists v,w\in\R \text{ s.t. } \\
	&v\ge y_1-y_2,v^2\le tx_1,w\ge y_2-y_1,w^2\le tx_2. &\bigg\}
	\end{align*}
While the representation of $\cl\conv(X_-)$ in Proposition~\ref{prop:otherX} is more compact than the one proposed in this paper, the extended formulation of $\cl\conv(Z_-)$ in Proposition~\ref{prop:otherZ} requires much more additional variables than the one we obtain above.
}
\end{remark}
	
	\ignore{
Both extended formulation and explicit form of $\cl\conv(Z_-)$ have been studied in . To make a comparison, we state the result appearing in \cite{atamturk2018signal} as follows.

\begin{proposition}[\citet{atamturk2018signal}]\label{prop:otherX}
	The convex hull of $X_-$ can be represented as 

\end{proposition}
%\todo{This comparison gives the impression that the approach is not effective\\ Shaoning: Done.}
%Notice that the extended formulation in Proposition~\ref{prop:others} requires two additional variables while ours requires six additional variables if we express it in the SOCP form. Thus, the extended formulation in Proposition~\ref{prop:others} is more compact.
Also, the authors discover a nice structural property of $\cl\conv(Z_-)$. Specifically, if we decompose the bivariate quadratic function as
\begin{align*}
d_1y_1^2-2y_1y_2+d_2y_2^2=&d_1(y_1-\frac{y_2}{d_1})^2+(d_2-\frac{1}{d_1})y_2^2\\
=&d_2(\frac{y_1}{d_2}-y_2)^2+(d_1-\frac{1}{d_2})y_2^2,
\end{align*}
and strengthen each quadratic term by adding valid inequality using perspective reformulation and applying Proposition~\ref{prop:otherX}, the resulting inequalities are sufficient to describe $\cl\conv(Z_-)$.  
\begin{proposition}[\citet{atamturk2018signal}]\label{prop:otherZ}
	A point $(x,y,t)\in\cl\conv(Z_-)$ if and only if $(x,y,t)\in[0,1]^2\times\R_+^3$ and there exists $t_1,t_2,q_1,q_2\in\R_+$ and $v_1.v_2,w_1,w_2\in\R_+$ such that the set of inequalities
	\begin{subequations}
		\begin{align*}
			y_1^2\le t_1x_1,\; y_2^2\le t_2x_2\\
			d_1v_1\ge d_1y_1-y_2,\;v_1^2\le q_1x_1\\
			d_1v_2\ge-d_1y_1+y_2,\;v_2^2\le 1_1x_2\\
			d_1q_1+t_2\left(d_2-\frac{1}{d_1}\right)\le t\\
			d_2w_1\ge y_1-d_2y_2,\;w_1^2\le q_2x_1\\
			d_2w_2\ge-y_1+d_2y_2,\;w_2^2\le q_2x_2\\
			d_2q_2+t_1\left(d_1-\frac{1}{d_2} \right)\le t
		\end{align*}
	\end{subequations}
is feasible.
\end{proposition}
\end{remark}

\todo{We may want to state the extended formulation from [5] for $Z_-$ for completeness instead of the effusive remark in nice structural results above. In [5] $Z_-$ and $X_-$ are interchanged, we may want to keep notation consistent to prevent confusion.}
}

\subsection{Description in the original space of variables $x,y,t$}\label{sec:convexHullDescription}
The purpose of this section is to express $\cl\conv(\set Z_+)$ and $\cl\conv(\set X_+)$ in the original space. 

Let $\Lambda:=\left\{\lambda\in\R:\max\{0,x_1+x_2-1\}\le\lambda\le \min\{x_1,x_2\}\right\}$, i.e., the set of
 feasible $\lambda$ implied by constraint \eqref{eq:posDisjunctA}. Define 
\[
G(\lambda,w):=\frac{d_1(y_1-\lambda w_1)^2}{x_1-\lambda}+\frac{d_2(y_2-\lambda w_2)^2}{x_2-\lambda}+\lambda(d_1w_1^2+2w_1w_2+d_2w_2^2)
\]
and $g:\Lambda\rightarrow\R$ as
\[g(\lambda):=\min\limits_{w\in \R_+^2} G(\lambda,w).\]
Note that as $G$ is SOCP-representable, it is convex.
We first prove an auxiliary lemma that will be used in the derivation.
\begin{lemma}\label{lem:monotonicity}
	Function $g$ is  non-decreasing over $\Lambda$.
\end{lemma}
\begin{proof}Note that for any fixed $w$ and $\lambda<\min\{ x_1,x_2 \}$, we have
	\begin{align*}
	\frac{\partial G(\lambda,w)}{\partial \lambda}= \ &\frac{d_1[2(\lambda w_1-y_1)w_1(x_1-\lambda)+(y_1-\lambda w_1)^2]}{(x_1-\lambda)^2}\\
	&+\frac{d_2[2(\lambda w_2-y_2)w_2(x_2-\lambda)+(y_2-\lambda w_2)^2]}{(x_2-\lambda)^2}\\
	&+(d_1w_1^2+2w_1w_2
	+d_2w_2^2)\\
	= \ &\frac{d_1[w_1^2(x_1-\lambda)^2+2(\lambda w_1-y_1)w_1(x_1-\lambda)+(y_1-\lambda w_1)^2]}{(x_1-\lambda)^2}\\
	&+\frac{d_2[w_2^2(x_2-\lambda_2)^2+2(\lambda w_2-y_2)w_2(x_2-\lambda)+(y_2-\lambda w_2)^2]}{(x_2-\lambda)^2}\\
	&+2w_1w_2\\
	= \ &\frac{d_1(w_1x_1-y_1)^2}{(x_1-\lambda)^2}+\frac{d_2(w_2x_2-y_2)^2}{(x_2-\lambda)^2}+2w_1w_2\\
	\ge \ &0.
	\end{align*}
	Therefore, for fixed $w$, $G(\cdot,w)$ is nondecreasing. Now for  $\tilde{\lambda}\le \hat{\lambda}$, let $\tilde w$ and $\hat w$ be 
	optimal solutions defining $g(\tilde{\lambda})$ and $g(\hat{\lambda})$. Then,
	\[g(\tilde\lambda)=G(\tilde\lambda,\tilde w)\le G(\tilde\lambda,\hat w)\le G(\hat\lambda,\hat w)=g(\hat\lambda),\]
	proving the claim.
	\ignore{
	\[  \min_{w\in\R_+^2}\;G(\tilde\lambda,w)\le\min_{w\in\R_+^2}\;G(\hat\lambda,w).\]
	
	 this is, $G(\tilde{\lambda},{w})\le G(\hat{\lambda},{w})$ holds for any $\tilde{\lambda}\le \hat{\lambda}$. Since $w$ is taken arbitrarily, it holds that 
	\[  \min_{w\in\R_+^2}\;G(\tilde\lambda,w)\le\min_{w\in\R_+^2}\;G(\hat\lambda,w).\]
	Namely, we find that $g(\tilde{\lambda})\le g(\hat{\lambda})$.
	\todo{************** For $\tilde{\lambda}$ and $\hat{\lambda}$ one may have different $w$'s optimal. Not sure how to compare $G(\tilde{\lambda},\tilde{w})$ and $G(\hat{\lambda},\hat {w})$\\ Shaoning: I add more details. It can also be obtained from $g(\tilde\lambda)=G(\tilde\lambda,\tilde w)\le G(\tilde\lambda,\hat w)\le G(\hat\lambda,\hat w)=g(\hat\lambda)$.}
}
\end{proof}
%To get the explicit form of $g(\lambda)$, we assume for simplicity that $(x,y,w)$ is nondegenerate, i.e. $0<x_1,x_2<1,y>0,w>0$, and the result is also applicablewhen inequaliteis are not necessarily strict.

We now state and prove the main result in this subsection.
\begin{proposition}\label{prop:originalPos}
	Define 
	\[ f(x,y,\lambda;d) := \frac{(d_1d_2-1)(d_1x_2y_1^2+d_2x_1y_2^2)+2\lambda d_1d_2y_1y_2+\lambda(d_1y_1^2+d_2y_2^2)}{(d_1d_2-1)x_1x_2-\lambda^2+\lambda(x_1+x_2)},\]
	and 
	\[f_+^*(x,y;d) :=\begin{cases}
	\frac{d_1y_1^2}{x_1}+\frac{d_2y_2^2}{x_2}&\text{if }x_1+x_2\le1\\
	\frac{d_2y_2^2}{x_2}+\frac{d_1y_1^2}{1-x_2}&\text{if }0\le x_1+x_2-1\le (x_1y_2-d_1x_2y_1)/y_2\\
	\frac{d_1y_1^2}{x_1}+\frac{d_2y_2^2}{1-x_1}&\text{if }0\le x_1+x_2-1\le (x_2y_1-d_2x_1y_2)/y_1\\
	f(x,y,x_1+x_2-1)&\text{o.w.}
	%	\frac{d_1d_2(d_1x_2y_1^2+d_2x_1y_2^2)+2d_1d_2y_1y_2(x_1+x_2-1)-d_1(1-x_1)y_1^2-d_2(1-x_2)y_2^2}{(d_1d_2-1)x_1x_2+x_1+x_2-1}&\text{o.w..}
	\end{cases} \]
	Then, the set $\cl\conv(\set Z_+)$ can be expressed as
	\[ \cl\conv(\set Z_+)=\{ (x,y,t)\in[0,1]^2\times\R^3_+:t\ge f_+^*(x_1,x_2,y_1,y_2;d_1,d_2) \}. \]	
\end{proposition}
\begin{proof}
	 First, observe that we may assume $x_1,x_2>0$, as otherwise
	 $x_1 + x_2 \le 1$ and $f^*_+$ reduces to the perspective function for the univeriate case.
%	  $0\le\lambda\le\min\{x_1,x_2\}=0$, which implies $\lambda=0$, and the conclusion follows trivially. 
	 To find the representation in the original space of variables, we first project out variables $z$ in Proposition~\ref{prop:extendedPos}. Specifically, notice that $g(\lambda)$ can be rewritten in the following form by letting $z_i=\lambda w_i,i=1,2$: 
	\begin{align}
	g(\lambda)=\min\;\;&\frac{d_1(y_1-z_1)^2}{x_1-\lambda}+\frac{d_2(y_2-z_2)^2}{x_2-\lambda}+\frac{d_1z_1^2+2z_1z_2+d_2z_2^2}{\lambda}\label{eq:gLambda}\\
	\text{s.t. }& z_i\ge0,\;\; i=1,2.\tag{$s_i$}
	\end{align}
	By Proposition~\ref{prop:extendedPos}, a point $(x,y,t)\in[0,1]^2\times\R_+^3$ belongs to $\cl\conv(\set Z_+)$ if and only if $t\ge \min_{\lambda\in \Lambda}\;g(\lambda)$.
	For given $\lambda\in \Lambda$, optimization problem \eqref{eq:gLambda} is convex with affine constraints, thus Slater condition holds. Hence, the following KKT conditions are necessary and sufficient for the minimizer:
	\begin{subequations}
		\begin{align}
		&\frac{2d_1}{x_1-\lambda}(z_1-y_1)+\frac{2(d_1z_1+z_2)}{\lambda}-s_1=0\label{eq:diff1}\\
		&\frac{2d_2}{x_2-\lambda}(z_2-y_2)+\frac{2(d_2z_2+z_1)}{\lambda}-s_2=0\label{eq:diff2}\\
		&z_1s_1=0\label{eq:complement1}\\
		&z_2s_2=0\label{eq:complement2}\\
		&s_i,z_i\ge0,\;\; i=1,2.
		\end{align}
	\end{subequations}
	Let us analyze the KKT system considering the positiveness of $s_1$ and $s_2$. 
	\begin{itemize}
		\item \textit{Case $s_1>0$.} By \eqref{eq:complement1}, $z_1=0$ and by \eqref{eq:diff1}, $z_2>0$, which implies $s_2=0$ from \eqref{eq:complement2}.  Hence, \eqref{eq:diff1} and \eqref{eq:diff2} reduce to
		\begin{align*}
		&\frac{2z_2}{\lambda}=\frac{2d_1}{x_1-\lambda}y_1+s_1\\
		&\frac{2d_2}{x_2-\lambda}(z_2-y_2)+\frac{2d_2z_2}{\lambda}=0.
		\end{align*}
		Solving these two linear equations, we get $z_2=\frac{y_2}{x_2}\lambda$ and $s_1=2(\frac{y_2}{x_2}-\frac{d_1y_1}{x_1-\lambda})$. This also indicates $s_1\ge 0$ iff $\lambda\le (x_1y_2-d_1x_2y_1)/y_2$. By replacing the variables with their optimal values in the objective function \eqref{eq:gLambda}, we find that
		\begin{subequations}\label{eq:explicitForm1}
			\begin{align}
			g(\lambda)=&\frac{d_1y_1^2}{x_1-\lambda}+\frac{d_2}{x_2-\lambda}\bigg (y_2-\frac{y_2}{x_2}\lambda \bigg )^2+\frac{d_2}{\lambda} \bigg (\frac{y_2}{x_2}\lambda \bigg)^2\\
%			=&\frac{d_1y_1^2}{x_1-\lambda}+\frac{d_2y_2^2(x_2-\lambda)}{x_2^2}+\frac{\lambda d_2y_2^2}{x_2^2}\\
			=&\frac{d_1y_1^2}{x_1-\lambda}+\frac{d_2y_2^2}{x_2}
			\end{align}
		\end{subequations}
		when $\lambda\in[0, (x_1y_2-d_1x_2y_1)/y_2]\cap \Lambda$.
		
		\item \textit{Case $s_2>0$.} 
		Similarly, we find that
		\begin{equation}\label{eq:explicitForm2}
		g(\lambda)=\frac{d_1y_1^2}{x_1}+\frac{d_2y_2^2}{x_2-\lambda}
		\end{equation}
		when $\lambda\in[0, (x_2y_1-d_2x_1y_2)/y_1]\cap \Lambda$.
		
		\item \textit{Case $s_1=s_2=0$.} 
		In this case, \eqref{eq:diff1} and \eqref{eq:diff2} reduce to
		\[\begin{pmatrix}
		d_1x_1&x_1-\lambda\\
		x_2-\lambda&d_2x_2
		\end{pmatrix}\begin{pmatrix}
		z_1\\z_2
		\end{pmatrix}=\lambda\begin{pmatrix}
		d_1y_1\\d_2y_2
		\end{pmatrix}. \]
		If $\lambda>0$, the determinant of the matrix is $(d_1d_2-1)x_1x_2+\lambda(x_1+x_2-\lambda)>0$ and the system has a unique solution. It follows that 
		%\todo{Is the matrix invertable?\\ Shaoning: At the beginning of the proof, I add the assumption that $x>0$. Under this assumption, we don't need to worry about the singularity of the matrix. }
		\[\begin{pmatrix}
		z_1\\z_2
		\end{pmatrix}=\lambda\begin{pmatrix}
		d_1x_1&x_1-\lambda\\
		x_2-\lambda&d_2x_2
		\end{pmatrix}^{-1}\begin{pmatrix}
		d_1y_1\\d_2y_2
		\end{pmatrix}, \]
		i.e.,
		\[ z_1=\frac{\lambda(d_1d_2x_2y_1+(\lambda-x_1)d_2y_2)}{(d_1d_2-1)x_1x_2-\lambda^2+\lambda(x_1+x_2)},\]
		\[ z_2=\frac{\lambda(d_1d_2x_1y_2+(\lambda-x_2)d_1y_1)}{(d_1d_2-1)x_1x_2-\lambda^2+\lambda(x_1+x_2)}.\]
		Therefore, the bounds $z_1, z_2\ge0$ imply lower bounds
		\[\lambda\ge (x_1y_2-d_1x_2y_1)/y_2, \quad \lambda\ge (x_2y_1-d_2x_1y_2)/y_1 \]
		on $\lambda$. Moreover, from \eqref{eq:diff1} and \eqref{eq:diff2}, we have \[ \frac{d_1(y_1-z_1)}{x_1-\lambda}=\frac{d_1z_1+z_2}{\lambda} \text{ and } \frac{d_2(y_2-z_2)}{x_2-\lambda}=\frac{d_2z_2+z_1}{\lambda} \cdot \]
		By substituting the two equalities in \eqref{eq:gLambda}, we find that
		\begin{align*}
%		g(\lambda) =& \frac{d_1z_1+z_2}{\lambda}(y_1-z_1)+\frac{d_2z_2+z_1}{\lambda}(y_2-z_2)+\frac{d_1z_1^2+2z_1z_2+d_2z_2^2}{\lambda}\\
	g(\lambda) 		=&\big (d_1y_1z_1+y_1z_2+d_2y_2z_2+y_2z_1 \big )/\lambda\\
%		=&\frac{1}{\lambda}\begin{pmatrix}
%		d_1y_1+y_2&d_2y_2+y_1
%		\end{pmatrix}\begin{pmatrix}
%		z_1\\z_2
%		\end{pmatrix}\\
%		=&\begin{pmatrix}
%		d_1y_1+y_2&d_2y_2+y_1
%		\end{pmatrix}\begin{pmatrix}
%		d_1x_1&x_1-\lambda\\
%		x_2-\lambda&d_2x_2
%		\end{pmatrix}^{-1}\begin{pmatrix}
%		d_1y_1\\d_2y_2
%		\end{pmatrix}\\
		=&  \frac{(d_1d_2-1)(d_1x_2y_1^2+d_2x_1y_2^2)+2\lambda d_1d_2y_1y_2+\lambda(d_1y_1^2+d_2y_2^2)}{(d_1d_2-1)x_1x_2-\lambda^2+\lambda(x_1+x_2)}.
%		=&f(x,y,\lambda).
		\end{align*}
		By the convention of division by $0$, the above discussion for this case is also applicable when $\lambda=0$. Therefore, 
		\begin{equation}\label{eq:nasty}
		g(\lambda)=f(x,y,\lambda;d)
		\end{equation} 
		when  $\lambda\in[ \max\{ (x_1y_2-d_1x_2y_1)/y_2,(x_2y_1-d_2x_1y_2)/y_1,+\infty \}\cap \Lambda ].$	
	\end{itemize}

	To see that the three pieces of $g(\lambda)$ considered above are, indeed, mutually exclusive, observe that
	when $\lambda\le(x_1y_2-d_1x_2y_1)/y_2$, this is, $\frac{y_2(x_1-\lambda)}{x_2y_1}\ge d_1 $, we have $\frac{d_2y_2}{y_1}\frac{x_1-\lambda}{x_2}\ge d_1d_2\ge1 $. Since $\frac{x_1-\lambda}{x_2}\frac{x_2-\lambda}{x_1}\le \frac{x_1}{x_2}\frac{x_2}{x_1}=1$, it holds $\frac{d_2y_2}{y_1}\frac{x_1-\lambda}{x_2}\ge\frac{x_1-\lambda}{x_2}\frac{x_2-\lambda}{x_1},$ that is, $\lambda\ge (x_2y_1-d_2x_1y_2)/y_1$. 
	
	Finally, notice when $\lambda =0$, \eqref{eq:explicitForm1}, \eqref{eq:explicitForm2}, and \eqref{eq:nasty} 
	reduce to
	\[ g(0)=\frac{d_1y_1^2}{x_1}+\frac{d_2y_2^2}{x_2} \cdot \] 
	By Lemma~\ref{lem:monotonicity}, $\min_{\lambda\in \Lambda}g(\lambda)=g(\max\{0,x_1+x_2-1 \})$.  Combining this fact with the above discussion, Proposition~\ref{prop:originalPos} holds.
\end{proof}

\begin{remark} 
For further intuition, we now comment on the validity of each piece of $t\ge f^*_+(x,y;d)$ over $[0,1]^2\times \R^3_+$ for $\set Z_+$. Because the first piece can be obtained by dropping the nonnegative cross product term $y_1y_2$ and then strengthening $t\ge y_1^2+y_2^2$ using perspective reformulation, it is valid everywhere. When $x_1+x_2<1$ and $y_1,y_2>0$, $t\ge y_i^2/x_i+y_j^2/(1-x_i)>f^*_+(x,y;1,1)$ for $i\neq j$. Therefore, the second and the third pieces are not valid on the domain  $[0,1]^2\times \R^3_+$. 
%We have the following remark for the validity of the last piece. 

	If $d_1d_2>1$, the last piece $t\ge f(x,y,x_1+x_2-1;d)$ is not valid for $\cl\conv(\set Z_+)$ everywhere, as seen by exhibiting a point $(x,y,t)\in\cl\conv(\set Z_+)$ violating $t\ge f(x,y,x_1+x_2-1;d)$. To do so, let
	\[ (x_1,x_2,y_1,y_2,t) = (0.5, \frac{1}{d_1d_2+1}+\epsilon, \frac{1}{\sqrt{d_1}},2\sqrt{d_1},f^*_+(x,y)), \]
	where $\epsilon>0$ is small enough so that $x_1+x_2<1$, i.e., $x_2<0.5$. With this choice, $f^*_+(x,y)=d_1y_1^2/x_1+d_2y_2^2/x_2.$ Let $\tilde{\lambda}=x_1+x_2-1$, then $\tilde\lambda(x_1+x_2)-\tilde\lambda^2=\tilde\lambda$. Hence, for point $(x,y,t)$, we have
	\begin{align*}
	f(x,y,\tilde\lambda;d)=&\frac{(d_1d_2-1)(d_1x_2y_1^2+d_2x_1y_2^2)+2\tilde\lambda d_1d_2y_1y_2+\tilde\lambda(d_1y_1^2+d_2y_2^2)}{(d_1d_2-1)x_1x_2+\tilde\lambda}\\
	=& \frac{(d_1d_2-1)x_1x_2(d_1y_1^2/x_1+d_2y_2^2/x_2)+\tilde\lambda (d_1y_1^2+2d_1d_2y_1y_2+d_2y_2^2)}{(d_1d_2-1)x_1x_2+\tilde\lambda}\\
	=&(1-\alpha)f_+^*(x,y)+\alpha(d_1y_1^2+2d_1d_2y_1y_2+d_2y_2^2),
	\end{align*}
	where $\alpha = \tilde\lambda/((d_1d_2-1)x_1x_2+\tilde\lambda)$. Since $\tilde\lambda<0$, $\alpha<0$ if and only if 
	\begin{align*}
	&(d_1d_2-1)x_1x_2+x_1+x_2-1>0\\	
	\iff&d_1d_2>\frac{(1-x_1)(1-x_2)}{x_1x_2}=\frac{1}{x_2}-1\quad\text{(by $x_1=0.5$)}\\
	\iff&x_2>\frac{1}{d_1d_2+1},
	\end{align*}
	which is true by the choice of $x_2$. Moreover,
	\begin{align*}
	&f^*_+(x,y)=d_1y_1^2/x_1+d_2y_2^2/x_2=2+8d_1d_2\\
	>\,&d_1y_1^2+2d_1d_2y_1y_2+d_2y_2^2=1+8d_1d_2.
	\end{align*}
	This indicates
	$f(x,y,\tilde\lambda;d)> (1-\alpha)f_+^*(x,y)+\alpha f_+^*(x,y)=f_+^*(x,y)=t,$ that is,
	$t\ge f(x,y,x_1+x_2-1;d)$ is violated.
\end{remark}

Observe that if $d_1 d_2 =1$, then $f(x,y,x_1+x_2-1;d)$ reduces to the original quadratic
$d_1 y_1^2 + 2y_1 y_2 + d_2y_2$. Otherwise, although $t\ge f(x,y,x_1+x_2-1;d)$ appears complicated,
the next proposition implies that it is convex over its restricted domain and can, in fact, 
be stated as an SDP constraint.
%\begin{lemma}
%	The following inequality is valid for $\cl\conv(Z_+)$
%	\[ t\ge f(x,y,\lambda;d),\exists\lambda\in \Lambda. \]
%\end{lemma}
%\begin{proof}
%	The conclusion follows from $$f(x,y,\lambda;d)=\min_{w\in\R^2}G(\lambda,w)\le \min_{w\in\R^2_+}G(\lambda,w)=g(\lambda)$$.
%\end{proof}

\begin{proposition}
	If $d_1 d_2 > 1$ and $x_1+x_2-1>0$, then $t\ge f(x,y,x_1+x_2-1;d)$ can be rewritten as the SDP constraint
	\[ \begin{pmatrix}
	{t}/({d_1d_2-1}) &  y_1 &y_2\\
	y_1 & d_2x_1+x_2/d_1-1/d_1 & {-x_1-x_2+1}\\
	y_2& {-x_1-x_2+1}&x_1/d_2+d_1x_2-1/d_2\\
	\end{pmatrix}\succeq 0. \]
\end{proposition}
\begin{proof}	
%	We first show the validity of the inequality $t\ge f(x,y,x_1+x_2-1;d)$. On the one hand, we have $f(x,y,\lambda;d)=\min_{w\in\R^2}G(\lambda,w)$, which implies if $x_1+x_2-1\ge0$, 
%	\[ f(x,y,x_1+x_2-1;d)=\min_{w\in\R^2}G(\lambda,x_1+x_2-1)\ge \min_{w\in\R_+^2}G(\lambda,x_1+x_2-1)=f^*_+(x,y). \]
%	On the other hand, 	
	Notice that for $\lambda=x_1+x_2-1> 0$, $f(x,y,\lambda;d)$ 
	can be rewritten in the form
	\[ f(x,y,\lambda;d)=\frac{1}{D}\hat{y}'A^*\hat{y}, \]
	where $D=(d_1d_2-1)x_1x_2+x_1+x_2-1>0, \hat{y}'=(\sqrt{d_1}y_1,\sqrt{d_2}y_2)$ and \[A^*=\begin{pmatrix}
	(d_1d_2-1)x_2+\lambda& \sqrt{d_1d_2}\lambda\\
	\sqrt{d_1d_2}\lambda& (d_1d_2-1)x_1+\lambda
	\end{pmatrix}. \]
	Observe $\det(A^*)=(d_1d_2-1)D$. 
	Hence,
	\[ f(x,y,\lambda;d)=\frac{(d_1d_2-1)}{\det(A^*)}\hat{y}^TA^*\hat{y}= (d_1d_2-1)\hat{y}^TA^{-1}\hat{y},\]
	where $A$ is the adjugate of $A^*$, i.e., 
	\[ A=\begin{pmatrix}
	(d_1d_2-1)x_1+\lambda& -\sqrt{d_1d_2}\lambda\\
	-\sqrt{d_1d_2}\lambda& (d_1d_2-1)x_2+\lambda
	\end{pmatrix} \cdot \] 
	Note that $A\succ0$. By Schur Complement Lemma, $t/(d_1d_2-1)\ge \hat{y}'A^{-1}\hat{y}$ if and only if 
	\[ \begin{pmatrix}
	t/(d_1d_2-1) & \hat y^T\\
	\hat y & A
	\end{pmatrix}\succeq 0, \]
	i.e.,
	\[ \begin{pmatrix}
	t/(d_1d_2-1) & \sqrt{d_1} y_1 &\sqrt{d_2}y_2\\
	\sqrt{d_1} y_1 & (d_1d_2-1)x_1+\lambda & -\sqrt{d_1d_2}\lambda\\
	\sqrt{d_2} y_2& -\sqrt{d_1d_2}\lambda& (d_1d_2-1)x_2+\lambda\\
	\end{pmatrix}\succeq 0, \]
	which is further equivalent to
	\[ \begin{pmatrix}
	t/(d_1d_2-1) &  y_1 &y_2\\
	y_1 & (d_1d_2-1)x_1/d_1+\lambda/d_1 & -\lambda\\
	y_2& -\lambda& (d_1d_2-1)x_2/d_2+\lambda/d_2\\
	\end{pmatrix}\succeq 0. \]
	The conclusion follows by taking $\lambda = x_1+x_2-1$.
\end{proof}

From Proposition~\ref{prop:originalPos}, we get the convex hull of rank-one case 
$\set X_+$ by setting $d_1 = d_2 = 1$. 
\begin{corollary}\label{prop:originalPosRank1}
	\[ \cl\conv(\set X_+)= \left\{ (x,y,t)\in[0,1 ]^2\times\R^3_+:t\ge f_{1+}(x,y)\right\},\]
	where 
	\[f_{1+}(x,y)=\begin{cases}
	\frac{y_1^2}{x_1}+\frac{y_2^2}{x_2}&\text{if }x_1+x_2\le1\\
	\frac{y_2^2}{x_2}+\frac{y_1^2}{1-x_2}&\text{if }0\le x_1+x_2-1\le (x_1y_2-x_2y_1)/y_2\\
	\frac{y_1^2}{x_1}+\frac{y_2^2}{1-x_1}&\text{if }0\le x_1+x_2-1\le (x_2y_1-x_1y_2)/y_1\\
	(y_1+y_2)^2&\text{o.w.}
	\end{cases} \]
\end{corollary}

One way to employ $f_{1+}$ to define valid inequalities for $\set Z_+$ is to
consider the two decompositions of the bivariate quadratic function given by
\begin{align*}
	d_1y_1^2+2y_1y_2+d_2y_2^2=& \; d_1(y_1+\frac{y_2}{d_1})^2+(d_2-\frac{1}{d_1})y_2^2\\
	=& \; d_2(\frac{y_1}{d_2}+y_2)^2+(d_1-\frac{1}{d_2})y_2^2.
\end{align*}
Applying perspective reformulation and Corollary~\ref{prop:originalPosRank1} to the separable and pairwise quadratic terms, respectively, one can obtain two simple valid inequalities for $\set Z_+$:
\begin{subequations}\label{eq:decomp}
	\begin{align}
		t\ge \ & d_1f_{1+}(x_1,x_2,y_1,\frac{y_2}{d_1})+(d_2-\frac{1}{d_1})\frac{y_2^2}{x_2}\\
		t\ge \ & d_2f_{1+}(x_1,x_2,\frac{y_1}{d_2},y_2)+(d_1-\frac{1}{d_2})\frac{y_1^2}{x_1}.
	\end{align}
\end{subequations}
\citet{atamturk2018signal} show that, for the complementary set $\set Z_-$, counterparts of \eqref{eq:decomp}  along with the bound constraints are sufficient to describe $\cl\conv(\set Z_-)$. However, the following example shows that this is not true for $\cl\conv(\set Z_+)$, highlighting
the more complicated structure of $\cl\conv(\set Z_+)$ compared to its complementary set $\cl\conv(\set Z_-)$.
\begin{example}\label{ex:difference}
	Consider $\set Z_+$ with $d_1=d_2=d=2$, and let $x_1=x_2=x=2/3$, $y_1=y_2=y>0$ and $t=f^*_+(x,y)$. Then $(x,y,t)\in\cl\conv(\set Z_+)$.  On the one hand, $x_1+x_2>1$ implies \[t=f^*_+(x,y)=f(2/3,2/3,y,y,1/3)=\frac{133}{11}y^2.\]
	On the other hand, $f_{1+}(x,x,y,y/d)=(y+y/d)^2=9/2y^2$ indicates that \eqref{eq:decomp} reduces to
	\[ t\ge \frac{27}{4}y^2. \]
	Since $\frac{133}{11}y^2>\frac{27}{4}y^2$, \eqref{eq:decomp} holds strictly at this point.\qed
	
\end{example}

For completeness, we finish this section, recalling the convex hull of $\set Z_-$ in the original space of variables as given in \citep{atamturk2018signal}: 
 $$\cl\conv(\set Z_-)=\left\{(x,y,t)\in [0,1]^2\times \R_+^3: t \ge f^*_-(x_1,x_2,y_1,y_2;d_1,d_2) \right\} $$
where
$$f^*_-(x,y;d)=\begin{cases}\frac{d_1y_1^2-2y_1y_2+y_2^2/d_1}{x_1}+\frac{y_2^2}{x_2}\left(d_2-\frac{1}{d_1}\right)&\text{if }x_1\geq x_2\text{ and }d_1y_1\geq y_2\\
\frac{d_1y_1^2-2y_1y_2+d_2y_2^2}{x_2}&\text{if }x_1\geq x_2\text{ and }d_1y_1\leq y_2\\
\frac{d_1y_1^2-2y_1y_2+d_2y_2^2}{x_1}&\text{if }x_1\leq x_2\text{ and }y_1\geq d_2y_2\\
\frac{y_1^2/d_2-2y_1y_2+d_2y_2^2}{x_2}+\frac{y_1^2}{x_1}\left(d_1-\frac{1}{d_2}\right)&\text{if }x_1\le x_2\text{ and }y_1\leq d_2y_2.\end{cases}$$

%--------------------------

\section{An SDP relaxation for \CQI }\label{sec:formulation}
%\subsection{A general scheme to obtain strong relaxations}

\ignore{
\todo{I don't think we should start with reviewing this paper. Let's directly describe what is new.}
\citet{dong2013} study the valid inequalities of the form
\begin{equation}\label{eq:QPB}
\innerProd{P}{Y}\ge \alpha'y+\delta'x+\gamma
\end{equation} 
for the box-constrained QP set $\cl\conv(S)$, where $P\succeq 0$, $\alpha,\delta\in\R^n$, $\gamma\in\R$ and 
\[ S:=\left\{ (x,y,Y):x\in\{ 0,1 \}^n,y\in[0,1]^n,Y=yy',y_i\le x_i\;\forall i \right\}. \]
In particular, \citet{dong2013} point out when a point $(\bar x,\bar y,\bar Y)$ satisfies all valid inequalities of type \eqref{eq:QPB} subject to that $P$ is a diagonal and PSD matrix, then $(\bar x,\bar y,\bar Y)\in \left\{ (x,y,Y):0\le Y_{ii}\le y_i\le x_i\le 1,Y_{ii}x_i\ge y_i^2\;\forall i\right\}.$ In other words, when $P$ is diagonal, adding the valid inequality of form \eqref{eq:QPB} is essentially equivalent to adding the perspective cuts. 
}

In this section, we will give an extended SDP relaxation for \CQI utilizing the convex hull results obtained in the previous section. Introducing a symmetric matrix variable $Y$, let us write \CQI as
 \begin{equation}\label{eq:mixed_integer_problem_SPD}
 \min\left\{ a'x+b'y+\innerProd{Q}{Y}:Y\succeq yy', (x,y) \in \I{n} \right\}.
 \end{equation}
Suppose for a class of PSD matrices $\Pi\subseteq \symM^n_+$ we have an underestimator $f_P(x,y)$ for $y'Py$ for any $P\in\Pi$.
%  i.e. $x'Px\ge f_P(x,y)$ holds over the feasible region. 
Then, since $\innerProd{P}{Y}\ge y'Py$, we obtain a valid inequality 
\begin{equation}\label{eq:nonlinearValid}
f_P(x,y)-\innerProd{P}{Y}\le0, \;P\in\Pi
\end{equation} 
for \eqref{eq:mixed_integer_problem_SPD}.
For example, if $\Pi$ is the set of diagonal PSD matrices and $f_P(x,y)=\sum_i P_{ii}y_i^2/x_i$, for $P\in\Pi$, then inequality \eqref{eq:nonlinearValid} is
the perspective inequality. 

Furthermore, since \eqref{eq:nonlinearValid} holds for any $P\in\Pi$, one can take the supremum over all $P\in\Pi$ to get an optimal valid inequality of the type \eqref{eq:nonlinearValid}
\begin{equation}\label{eq:optimalNonlinearValid}
\sup_{P\in\Pi}f_P(x,y)-\innerProd{P}{Y}\le 0.
\end{equation}

Back to the example of perspective reformulation, inequality \eqref{eq:optimalNonlinearValid} becomes
\[ \sup_{P\succeq0\text{ diagonal}} \left \{ \sum_i P_{ii}
\bigg ( y_i^2/x_i-Y_{ii} \bigg ) \right \} \le 0, \]
which can be further reduced to  the closed form $y_i^2\le Y_{ii}x_i,\forall i\in [n]$. Thus, the optimal perspective formulation \optPersp can be regarded as a special case obtained via this scheme.  
%\todo{What is the difference for the perspective case? Diving by $P_{ii}$ we get the same inequality for all diagonal $P$ with or without taking supremum.\\ Shaoning: Notice we have a $\sum_i$ here. Since $P_{ii}$'s are not identical, we can not divide the inequality by $P_{ii}$. Since (20) holds for any $P$, we must have each term $y_i^2/x_i-Y_{ii}\le0$, that is optimal persp.}

\ignore{
%\subsection{Realization of the general scheme}
 We can directly follow the above paradigm to develop a novel relaxation for \CQI %\eqref{eq:mixed_integer_problem} 
 by incorporating the information gained in Proposition~\ref{prop:extendedPos} and Remark~\ref{remark:extendedNeg}. In our case,
}

 Letting $\Pi$ be the class of $2\times 2$ PSD matrices and $f_P(\cdot)$ as the function describing the convex hull of the mixed-integer epigraph of $y'Py$, one can derive new valid inequalities for \CQI.
  Specifically, using the extended formulations for $f_+^*(x,y;d)$ and
   $f_-^*(x,y;d)$ describing $\cl\conv(\set Z_+)$ and $\cl\conv(\set Z_-)$, we have
\begin{subequations}\label{eq:extendedPos}	
	\begin{align}
	f_+^*(x,y;d)=\min_{z,\lambda}& \frac{d_1(y_1-z_1)^2}{x_1-\lambda}+\frac{d_2(y_2-z_2)^2}{x_2-\lambda}+\frac{d_1z_1^2+2z_1z_2+d_2z_2^2}{\lambda}\\
	\text{s.t.} & \ z_1\ge0,z_2\ge0\\
	& \ \max\{ 0,x_1+x_2-1\} \le\lambda\le\min\{ x_1,x_2 \},
	\end{align}
\end{subequations}
and 
\begin{subequations}\label{eq:extendedNeg}	
	\begin{align}
	f_-^*(x,y;d)=\min_{z,\lambda}& \frac{d_1(y_1-z_1)^2}{x_1-\lambda}+\frac{d_2(y_2-z_2)^2}{x_2-\lambda}+\frac{d_1z_1^2-2z_1z_2+d_2z_2^2}{\lambda}\\
	\text{s.t.} & \ z_1\le y_1,z_2\le y_2\\
	& \ \max\{ 0,x_1+x_2-1\} \le\lambda\le\min\{ x_1,x_2 \}.
	\end{align}
\end{subequations}
Since any $2 \times 2$ symmetric PSD matrix $P$ can be rewritten in the form of $P=p\begin{psmallmatrix}
d_1&1\\1&d_2
\end{psmallmatrix}$ or $P=p\begin{psmallmatrix}
d_1&-1\\-1&d_2
\end{psmallmatrix},$ we can take $f_P(x,y)=pf_+^*(x,y;d)$ or $f_P(x,y)=pf_-^*(x,y;d)$, correspondingly. Since we have the explicit form of $f_+^*(\cdot)$ and $f_-^*(\cdot)$, for any fixed $d$, \eqref{eq:nonlinearValid} gives a nonlinear valid inequality which can be added to \eqref{eq:mixed_integer_problem_SPD}.  Alternatively, \eqref{eq:extendedPos} and \eqref{eq:extendedNeg} can be used to reformulate these inequalities as conic quadratic inequalities in an extended space. Moreover, maximizing the inequalities gives the optimal valid inequalities among the class of of $2\times 2$ PSD matrices stated below. Recall that $\D := \{d \in \R^2: d_1\ge0, d_2\ge0, d_1d_2\ge1\}$. 

\begin{proposition} For any pair of indices $i<j$,
	the following inequalities are valid for \CQI:%\eqref{eq:mixed_integer_problem}.
	\begin{subequations}\label{eq:relaxationFormF}
		\begin{align}
		&\max_{d \in \D}\left\{ f_+^*(x_i,x_j,y_i,y_j;d_1,d_2)\!-d_1Y_{ii}\!-d_2Y_{jj}\!-2Y_{ij} \right\}\le0, \ \label{eq:relaxationPos}\\
		&\max_{d \in \D}\left\{ f_-^*(x_i,x_j,y_i,y_j;d_1,d_2)\!-d_1Y_{ii}\!-d_2Y_{jj}\!+2Y_{ij} \right\}\le0. \  \label{eq:relaxationNeg}
		\end{align}
	\end{subequations}	
\end{proposition}

\ignore{ From the above discussion, validity is clear. No need for a proof.
\begin{proposition}
	The following formulation is a valid convex relaxation of \CQI. %\eqref{eq:mixed_integer_problem}.
	\begin{subequations}\label{eq:relaxationFormF}
		\begin{align}
		\min \;& a'x+b'y+\innerProd{Q}{Y}\\
		\text{s.t.} \;&Y-yy'\succeq 0\\
		&\max_{d \in D}\left\{ f_+^*(x_i,x_j,y_i,y_j;d_1,d_2)\!-d_1Y_{ii}\!-d_2Y_{jj}\!-2Y_{ij} \right\}\le0, \ \forall i<j\label{eq:relaxationPos}\\
		&\max_{d \in D}\left\{ f_-^*(x_i,x_j,y_i,y_j;d_1,d_2)\!-d_1Y_{ii}\!-d_2Y_{jj}\!+2Y_{ij} \right\}\le0, \ \forall i<j\label{eq:relaxationNeg}\\
		& 0\le x_i\le 1, \ \forall i 
		\end{align}
	\end{subequations}	
\end{proposition}
\begin{proof}
	Note that \eqref{eq:relaxationPos} is the point supremum of convex functions and there convex.
	For any feasible solution to \CQI. %eqref{eq:mixed_integer_problem}. 
	Let $Y=yy'$. For any $i,j$ such that $i>j$ and any $d_1,d_2\ge0$ such that $d_1d_2\ge1$,
	\begin{align*}
	&f_+^*(x_i,x_j,y_i,y_j,d)-d_1Y_{ii}-d_2Y_{jj}-2Y_{ij}\\
	=&f_+^*(x_i,x_j,y_i,y_j,d)-d_1y_i^2-d_2y_j^2-2y_iy_j
	\le0,
	\end{align*}
	where the last inequality follows from Proposition~\ref{prop:extendedPos}. It implies that \eqref{eq:relaxationPos} is valid for problem \CQI. 
	Using the same argument, we can show \eqref{eq:relaxationNeg} is also a convex constraint and valid for \CQI.
\end{proof}
}

Optimal inequalities \eqref{eq:relaxationFormF} may be employed effectively if they can be expressed explicitly. 
We will now show how to write inequalities \eqref{eq:relaxationFormF} explicitly using an auxiliary $3 \times 3$ matrix variable $W$.

%Before proving  Proposition~\ref{prop:sdppos}, we first  show how to reformulate \eqref{eq:relaxationPos} and \eqref{eq:relaxationNeg} as SDP constraints via conic duality. 
\begin{lemma} \label{lemma:Wplus}
	%Point $(x_1,x_2,y_1,y_2,Y_{11},Y_{12},Y_{22})$ satisfying \eqref{eq:relaxationPos} is equivalent to the point is the feasible solution of the following inequality system	
	A point $(x_1,x_2,y_1,y_2,Y_{11},Y_{12},Y_{22})$ %$(x,y,Y)$ 
	satisfies inequality \eqref{eq:relaxationPos} if and only if there exists $W^+\in \symM^3_+$ such that the inequality system
	\begin{subequations}\label{eq:positiveRelaxation}
		\begin{align}
		&W^+_{12}\le Y_{12}\\
		&(Y_{11}-W^+_{11})(x_1-W^+_{33})\ge(y_1-W^+_{31})^2,W^+_{11}\le Y_{11},W^+_{33}\le x_1\\
		&(Y_{22}-W^+_{22})(x_2-W^+_{33})\ge(y_2-W^+_{32})^2,W^+_{22}\le Y_{22},W^+_{33}\le x_2\\
		&W^+_{31}\ge0,W^+_{32}\ge0\\
		&W^+_{33}\ge x_1+x_2-1
		\end{align}
	\end{subequations}
is feasible.
\end{lemma}

%We can also consider the SOCP optimization problem associated with $\cl\conv(Z_-)$ and use the trick to make \eqref{eq:relaxationNeg} a SDP constraint.
\begin{lemma} \label{lemma:Wneg}
		A point $(x_1,x_2,y_1,y_2,Y_{11},Y_{12},Y_{22})$ %$(x,y,Y)$ 
		satisfies inequality \eqref{eq:relaxationNeg} if and only if there exists $W^-\in \symM^3_{+}$ such that the inequality system
	\begin{subequations}\label{eq:negativeRelaxation}
		\begin{align}
		&Y_{12}\le W^-_{12}\\
		&(Y_{11}-W^-_{11})(x_1-W^-_{33})\ge(y_1-W^-_{31})^2,W^-_{11}\le Y_{11},W^-_{33}\le x_1\\
		&(Y_{22}-W^-_{22})(x_2-W^-_{33})\ge(y_2-W^-_{32})^2,W^-_{22}\le Y_{22},W^-_{33}\le x_2\\
		&W^-_{31}\le y_1,W^-_{32}\le y_2\\
		&W^-_{33}\ge x_1+x_2-1
		\end{align}
	\end{subequations}
is feasible.
\end{lemma}

\begin{proof}[Proof of Lemma~\ref{lemma:Wplus}]
	\ignore{
	Notice that the left hand of \eqref{eq:relaxationPos} is in fact a max-min-type optimization problem. If we can convert it to a min-min-type optimization problem, we can directly drop min-min operator in the constraint. In order to achieve this goal, we first take the SOCP duality with respect to the inner optimization problem to make \eqref{eq:relaxationPos} a max-max-type optimization problem, then we take the SDP duality with respect to the whole optimization problem and get the desired form.
}
    Writing $f^*_+$ as a conic quadratic minimization problem as in \eqref{eq:extendedPos},
	we first express inequality \eqref{eq:relaxationPos} as 
	%\textbf{max-min$\Rightarrow$max-max}\\
	%	Consider the optimization problem \eqref{eq:extendedPos}, which can be reformulated as a SOCP program and it implies 
		\begin{align*}
	%	0\ge\max_{d_1d_2\ge1,d_1\ge0,d_2\ge0}&\left\{ f_+^*(x_1,x_2,y_1,y_2,d_1,d_2)-d_1Y_{11}-d_2Y_{22}-2Y_{12} \right\}\\
		0 \ge \max_{d \in \D}\min_{t,\lambda,z}\quad & d_1t_1+d_2t_2+t_3-d_1Y_{11}-d_2Y_{22}-2Y_{12}\\
		\text{s.t.}\quad&t_1(x_1-\lambda)\ge(y_1-z_1)^2,t_1\ge0,x_1-\lambda\ge0\\
		&t_2(x_2-\lambda)\ge (y_2-z_2)^2,t_2\ge0,x_2-\lambda\ge0\\
		&\lambda t_3\ge \norm{2}{B_+z}^2,\lambda\ge0,t_3\ge0\\
		& \lambda \ge x_1 +x_2-1\\
		&z_1,z_2\ge0,
%		&z_2\ge0,
		\end{align*}
		where $B_+^2=\begin{psmallmatrix}
		d_1&1\\1&d_2
		\end{psmallmatrix}.$ 
		Taking the dual of the inner minimization, the inequality can be written as		
			\begin{align*}
	0\ge\max_{d \in \D} \max_{\alpha,\tau,\eta,\gamma,s,r} &\;-	\sum_{i=1,2} (x_i s_i + 2 y_i \eta_i)
	%-x_1s_1-x_2s_2-2y_1\eta_1-2y_2\eta_2
	+(x_1+x_2-1)\alpha-d_1Y_{11}-d_2Y_{22}-2Y_{12}\\
	\text{s.t. }&d_is_i\ge\eta_i^2, \;i=1,2\\
%	&d_2s_2\ge \eta_2^2\\
	&s_3\ge\norm{2}{\gamma}^2 \\
	&r_1, r_2, \alpha \ge0\\
	&\alpha+s_3=s_1+s_2\\
	&\begin{pmatrix}
	r_1-2\eta_1\\r_2-2\eta_2
	\end{pmatrix}=2B_+\begin{pmatrix}
	\gamma_1\\\gamma_2,
	\end{pmatrix},
	\end{align*}
	\ignore{
		\begin{align*}
		0\ge\max_{d \in D} \max_{\tau,\eta,\gamma,s,r} &\; 
		\sum_{i=1,2} -x_i s_i - 2 y_i \eta_i
%		-x_1s_1-x_2s_2-2y_1\eta_1-2y_2\eta_2
		+(x_1+x_2-1)\alpha-d_1Y_{11}-d_2Y_{22}-2Y_{12}\\
		\text{s.t. }&\tau_1s_1\ge \eta_1^2,\tau_1\ge0,s_1\ge0\\
		&\tau_2s_2\ge \eta_2^2, \\ %\tau_2\ge0,s_2\ge0\\
		&\tau_3s_3\ge\norm{2}{\gamma}^2, \\ %\tau_3\ge0,s_3\ge0\\
		&r_1\ge0\\
		&r_2\ge0\\
		&\alpha\ge0\\
		&\tau_1=d_1\\
		&\tau_2=d_2\\
		&\tau_3=1\\
		&\alpha+s_3=s_1+s_2\\
		&\begin{pmatrix}
		r_1-2\eta_1\\r_2-2\eta_2
		\end{pmatrix}=2B_+\begin{pmatrix}
		\gamma_1\\\gamma_2,
		\end{pmatrix},
		\end{align*}
	}
	where the last equation implies $\gamma = B_+^{-1}(r/2-\eta)$. Substituting out $\gamma$ and $s_3$, 
	and letting $u_i=\eta_i-r_i/2,i=1,2$,
	the maximization problem is further reduced to
		\begin{subequations}\label{eq:positiveConic}
			\begin{align*}
			0\ge\max_{d \in \D}&\max_{\alpha,\eta,s,r, u}
				-	\sum_{i=1,2} (x_i s_i + 2 y_i \eta_i)
			%-x_1s_1-x_2s_2-2y_1\eta_1-2y_2\eta_2
			+(x_1+x_2-1)\alpha-d_1Y_{11}-d_2Y_{22}-2Y_{12}\\
			\text{s.t. }&d_is_i\ge\eta_i^2, \;i=1,2\\
%			&d_2s_2\ge\eta_2^2\\
			&\eta_i\ge u_i,\;i=1,2\\
            &\alpha\ge0\\
%			&r_1, r_2, \alpha \ge0\\
%			&r_2\ge0\\
%			&\alpha\ge0\\
			&s_1+s_2-\alpha\ge
%			\begin{pmatrix}
%			r_1/2-\eta_1\\r_2/2-\eta_2
%			\end{pmatrix}'
u'
			\begin{bsmallmatrix}
			d_1&1\\1&d_2
			\end{bsmallmatrix}^{-1}
			u.
%			\begin{pmatrix}
%			r_1/2-\eta_1\\r_2/2-\eta_2
%			\end{pmatrix}\label{eq:shur1}.
			\end{align*}
		\end{subequations}
		 Applying Schur Complement Lemma to the last inequality, we reach
		\begin{subequations}\label{eq:firststepOpt}
			\begin{align*}
			2Y_{12}\ge\max_{\eta,s,u,r,d}\;& 	-	\sum_{i=1,2} (x_i s_i + 2 y_i \eta_i)
			%-x_1s_1-x_2s_2-2y_1\eta_1-2y_2\eta_2
			+(x_1+x_2-1)\alpha-d_1Y_{11}-d_2Y_{22}\\
			\text{s.t. }&d_is_i\ge\eta_i^2, \;i=1,2\\
	%		&d_2s_2\ge\eta_2^2\\
			&\eta_i\ge u_i,\;i=1,2\\
			&\alpha\ge0\\
			&\begin{pmatrix}
			s_1+s_2-\alpha&u_1&u_2\\
			u_1&d_1&1\\
			u_2&1&d_2
			\end{pmatrix}\succeq0.%\label{eq:firststepA}
			\end{align*}
		\end{subequations}
		Note the SDP constraint implies $d \in D$. Finally, taking the SDP dual of
		the maximization problem we arrive at
	\begin{subequations}
		\begin{align*}
		2Y_{12}\ge\min_{p,q,w,v,W^+} &\;2W^+_{12}\\
		\text{s.t. } & p_iq_i\ge w_i^2, \ p_i, q_i\ge0, \ i=1,2\\
		&v_i \ge0, \ i=1,2\\
%		&p_1q_1\ge w_1^2,p_1\ge0,q_1\ge0\\
%		&p_2q_2\ge w_2^2,p_2\ge0,q_2\ge0\\
%		&v_1\ge0,v_2\ge0\\
%		&W^+\succeq0\\
		&q_i+W^+_{33}=x_i,\ i=1,2 \\
		&p_i+W^+_{ii}=Y_{ii}, \ i=1,2 \\
		&2w_i+v_i=2y_i,  \ i=1,2 \\
%		&q_1+W^+_{33}=x_1,q_2+W^+_{33}=x_2\\
%		&p_1+W^+_{11}=Y_{11},p_2+W^+_{22}=Y_{22}\\
%		&2w_1+v_1=2y_1,2w_2+v_2=2y_2\\
		&2W^+_{3i}= v_i, \ i=1,2 \\
%		&-v_1+2W^+_{31}=0\\
%		&-v_2+2W^+_{32}=0\\
		&\beta-W^+_{33}=1-x_1-x_2\\
		&\beta\ge0, \ W^+\succeq0.
		\end{align*}
	\end{subequations}
	Substituting out $p,q,w,v,\beta$, we arrive at \eqref{eq:positiveRelaxation}.
\end{proof}

The proof of Lemma~\ref{lemma:Wneg} is similar and is omitted for brevity. Since both
\eqref{eq:relaxationPos} and \eqref{eq:relaxationNeg} are valid, 
using \eqref{eq:positiveRelaxation} and \eqref{eq:negativeRelaxation} together,  
one can obtain an SDP relaxation of \CQI. While inequalities in \eqref{eq:positiveRelaxation} and \eqref{eq:negativeRelaxation} are quite similar, in general, $W^+$ and $W^-$ do not have to coincide. However, we show below that choosing $W^+ = W^-$, the resulting SDP formulation is still valid %\eqref{eq:mixed_integer_problem}, \eqref{eq:rankTwoWithNonnegativeVariables}, 
and it is at least as strong as the strengthening obtained by valid inequalities \eqref{eq:relaxationFormF}. 

%We show now how to get an SDP relaxation for \CQI.
%Our goal is to prove the following proposition.
%\todo{where do we use the assumption $Q$ is psd?}

Let $\W$ be the set of points  $(x_1,x_2,y_1,y_2,Y_{11},Y_{12},Y_{22})$ such that there exists a $3 \times 3$ matrix $W$ satisfying 
\begin{subequations}\label{eq:sdppos}
\begin{align}
	&W_{12}= Y_{12} \label{eq:sdpposW12}\\
	&(Y_{11}-W_{11})(x_1-W_{33})\ge(y_1-W_{31})^2,W_{11}\le Y_{11},W_{33}\le x_1   \label{eq:sdpposW11}\\
	&(Y_{22}-W_{22})(x_2-W_{33})\ge(y_2-W_{32})^2,W_{22}\le Y_{22},W_{33}\le x_2  \label{eq:sdpposW22}\\
	&0 \le W_{31}\le y_1, \, 0 \le W_{32}\le y_2 \\ 
	&W_{33}\ge x_1+x_2-1 \\ 
	&W\succeq0 
\end{align}
\end{subequations}
Then, using \W for every pair of indices, we can define the strengthened SDP formulation
	\begin{subequations}\label{eq:rankTwoWithNonnegativeVariables}
	\begin{align}
	\min\; &a'x+b'y+\innerProd{Q}{Y}&\\
	(\sdppos) \ \ \		\text{s.t. }&Y-yy'\succeq0&\\\
	&(x_i, x_j, y_i, y_j, Y_{ii}, Y_{ij}, Y_{jj}) \in \W&\forall i< j \label{eq:inW} \\
	&	0 \le x \le 1, \ y \ge 0.
	\end{align}
\end{subequations}

\begin{proposition}\label{prop:sdppos}
	\sdppos \ is a valid convex relaxation of \CQI and every feasible solution to it satisfies all valid inequalities \eqref{eq:relaxationFormF}.
\end{proposition}

%\ignore{
%\begin{proposition}\label{prop:sdppos}
%	The following formulation is a valid convex relaxation of \CQI and every feasible solution to it satisfies all valid inequalities \eqref{eq:relaxationFormF}.
%	\begin{subequations}\label{eq:rankTwoWithNonnegativeVariables}
%		\begin{align}
%		\min\; &a'x+b'y+\innerProd{Q}{Y}&\\
%		\text{s.t. }&Y-yy'\succeq0&\\\
%		&W^{(ij)}_{12}= Y_{ij}&\forall i< j\\
%		&(Y_{ii}-W^{(ij)}_{11})(x_i-W^{(ij)}_{33})\ge(y_i-W^{(ij)}_{31})^2,W^{(ij)}_{11}\le Y_{ii},W^{(ij)}_{33}\le x_i&\forall i< j \label{eq:sdpposi}\\
%		&(Y_{jj}-W^{(ij)}_{22})(x_j-W^{(ij)}_{33})\ge(y_j-W^{(ij)}_{32})^2,W^{(ij)}_{22}\le Y_{jj},W^{(ij)}_{33}\le x_j&\forall i< j\label{eq:sdpposj}\\
%		&W^{(ij)}_{31}\ge0,W^{(ij)}_{32}\ge0&\forall i< j\label{eq:omega2}\\
%		&W^{(ij)}_{31}\le y_i,W^{(ij)}_{32}\le y_j&\forall i< j\label{eq:omega3}\\
%		&W^{(ij)}_{33}\ge x_i+x_j-1&\forall i< j\label{eq:omega4}\\
%		&W^{(ij)}\succeq0&\forall i< j\label{eq:omega1} \\
%		& 0\le x_i\le 1 & \forall i\label{eq:omega5}
%		\end{align}
%	\end{subequations}
%\end{proposition}
%}

\begin{proof}%[Proof of Proposition~\ref{prop:sdppos}]
	To see that \sdppos \ is a valid relaxation, consider a feasible solution $(x,y)$ of \CQI and let $Y=yy'$. For $i<j$, if $x_i=x_j=1$, constraint \eqref{eq:inW} is satisfied with $W=\begin{psmallmatrix}
	Y_{ii}&Y_{ij}&{y_i}\\
	Y_{ij}&Y_{jj}&y_j\\
	y_i&y_j&1
	\end{psmallmatrix}.$ Otherwise, without loss of generality, one may assume $x_i=0$. It follows that $Y_{ii}=y_i^2=Y_{ij}=y_iy_j=0$. Then, constraint \eqref{eq:inW} is satisfied with $W=0$.
 Moreover, if $W$ satisfies \eqref{eq:inW}, then $W$ satisfies \eqref{eq:positiveRelaxation} and \eqref{eq:negativeRelaxation} simultaneously.
 % which implies \eqref{eq:rankTwoWithNonnegativeVariables} is at least as strong as \eqref{eq:relaxationFormF}.
\end{proof}

%We refer to formulation \eqref{eq:rankTwoWithNonnegativeVariables} as \sdppos. 

\section{Comparison of convex relaxations}\label{sec:comparison}
In this section, we compare the strength of \sdppos \ with other convex relaxations of \CQI.
The perspective relaxation and the optimal perspective relaxation \optPersp \ for \CQI are well-known. A summary of the comparisons of strength is represented in Figure~\ref{fig:formulationRelation} in Section~\ref{sec:intro}.

%\subsection{Optimal perspective relaxations \& Shor's SDP} 
%We now show \optPersp is weaker than  \sdppos.
\begin{proposition}\label{prop:optPersp}
 \sdppos \ is at least as strong as  \optPersp and  \sdpstandard.
\end{proposition}
\begin{proof}
	By Theorem~\ref{thm:equivalence}, it is sufficient to show the statement for \optPersp.
	Note that \eqref{eq:inW} includes constraints
	\[
	\begin{pmatrix}
	Y_{ii}&y_i\\y_i&x_i
	\end{pmatrix}\succeq 
	\begin{pmatrix}
	W_{11}&W_{31}\\ W_{31}& W_{33}
\end{pmatrix}\succeq 0,
\] 
corresponding to \eqref{eq:sdpposW11}-\eqref{eq:sdpposW22}. 
Thus, the perspective constraints $Y_{ii} x_i \ge y_i^2$ are implied.
\end{proof}

In the context of linear regression,
%\subsection{Convex relaxation using rank-one with free continuous variables}
\citet{atamturk2019rank} study the convex hull of the epigraph of rank-one quadratic with indicators
\[ \set X_f=\left\{ (x,y,t)\in \{0,1\}^n \times\R^{n+1} :t\ge \bigg (\sum_{i=1}^n y_i \bigg )^2, \ y_i(1-x_i)=0, i \in [n] \right\}, \]
where the continuous variables are unrestricted in sign. Their extended SDP formulation
based on $\cl\conv(\set X_f)$, leads to the following relaxation for \CQI
\begin{subequations}\label{eq:rank-one}%\tag{\sdpfree}
\begin{align}
\min\;& a'x+b'y+\langle Q,Y\rangle\\
\text{s.t.}\;& Y - yy' \succeq 0 \\
&y_i^2\leq Y_{ii}x_i&\forall i\\
(\sdpfree) \ \ \  &\begin{pmatrix}x_i+x_j& y_i &y_j\\
y_i & Y_{ii} & Y_{ij}\\
y_j & Y_{ij}& Y_{jj}\end{pmatrix}\succeq 0, &\forall i < j \label{eq:sdpf:pair}\\
&y \ge 0, \ 0\le x\le 1.
\end{align}
\end{subequations} 

% \citet{atamturk2019rank} show \sdpfree is stronger than \optPersp. Thus, by Theorem~\ref{thm:equivalence}, \sdpfree is stronger than \sdpstandard \, as well. 
With the additional constraints \eqref{eq:sdpf:pair}, 
it is immediate that \sdpfree \ is stronger than \optPersp.
The following proposition compares \sdpfree and \sdppos. 
\begin{proposition}
	\sdppos \ is at least as strong as \sdpfree.
\end{proposition}
\begin{proof}
	\ignore{
	If we relax \eqref{eq:sdpposi} and \eqref{eq:sdpposj} by replacing $x_i$ and $x_j$ with $\max\{x_i,x_j \}$, then we get a valid but weaker relaxation, which is equivalent to the following formulation by eliminating $W^{(ij)}$

	Relaxing constraints
	\[
		\begin{pmatrix}
	Y_{ii}&y_i\\y_i&x_i
	\end{pmatrix}\succeq 0, \ \ 
			\begin{pmatrix}
	Y_{ii}&y_i\\y_i&x_i
	\end{pmatrix} \succeq 0, \ \ 
	\]
}

It suffices to show that for each pair $i<j$, constraint \eqref{eq:inW} of \sdppos implies \eqref{eq:sdpf:pair} of \sdpfree.
Rewriting \eqref{eq:sdpposW11}--\eqref{eq:sdpposW22}, we get
\[ W_{11}\le Y_{11}-\frac{(y_1-W_{31})^2}{x_1-W_{33}}, \; W_{22}\le Y_{22}-\frac{(y_2-W_{32})^2}{x_2-W_{33}} \cdot \]
Combining the above and \eqref{eq:sdpposW12} to substitute out $W_{11},W_{22}$ and $W_{12}$ in $W\succeq 0$, we arrive at
\[\begin{pmatrix}
Y_{11}-\frac{(y_1-W_{31})^2}{x_1-W_{33}}&Y_{12}&W_{31}\\
Y_{12}&Y_{22}-\frac{(y_2-W_{32})^2}{x_2-W_{33}}&W_{32}\\
W_{31}&W_{32}&W_{33}
\end{pmatrix} \succeq0,\; W_{33}\le x_1,\; W_{33} \le x_2,\]
which is equivalent to the following matrix inequality by Shur Complement Lemma
\[ \begin{pmatrix}
Y_{11}&Y_{12}&W_{31}&y_1-W_{31}&0\\
Y_{12}&Y_{22}&W_{32}&0&y_2-W_{32}\\
W_{31}&W_{32}&W_{33}&0&0\\
y_1-W_{31}&0&0&x_1-W_{33}&0\\
0&y_2-W_{32}&0&0&x_2-W_{33}
\end{pmatrix} \succeq0. \]
By adding the third row/column to the forth row/column and then adding the forth row/column to the fifth row/column, the large matrix inequality can be rewritten as
\[ \begin{pmatrix}
Y_{11}&Y_{12}&W_{31}&y_1&y_1\\
Y_{12}&Y_{22}&W_{32}&W_{32}&y_2\\
W_{31}&W_{32}&W_{33}&W_{33}&W_{33}\\
y_1&W_{32}&W_{33}&x_1&x_1\\
y_1&y_2&W_{33}&x_1&x_1+x_2-W_{33}
\end{pmatrix} \succeq0. \]
Because $W_{33}\ge0$, it follows that  
\[\begin{pmatrix}
Y_{11}&Y_{12}&y_1\\
Y_{12}&Y_{22}&y_2\\
y_1&y_2&x_1+x_2
\end{pmatrix}\succeq \begin{pmatrix}
Y_{11}&Y_{12}&y_1\\
Y_{12}&Y_{22}&y_2\\
y_1&y_2&x_1+x_2-W_{33}
\end{pmatrix}\succeq 0.\]
Therefore, constraints \eqref{eq:sdpf:pair} are implied by \eqref{eq:inW}, 
proving the claim.
%Replacing $x_i$ and $x_j$ in \eqref{eq:inW} with $\max\{x_i,x_j \}$, 
%we get a valid but weaker relaxation, which is equivalent to the following formulation after 
%eliminating $W$
%	\begin{subequations}\label{eq:compareWithRankOne}
%		\begin{align}
%		\min\; &a'x+b'y+\innerProd{Q}{Y}&\\
%		\text{s.t. }&y_i^2\le Y_{ii}x_i&\forall i\\
%		&Y-yy'\succeq0&\\
%		&\begin{pmatrix}
%		\max\{x_i,x_j\}&y_i&y_j\\
%		y_i&Y_{ii}&Y_{ij}\\
%		y_j&Y_{ij}&Y_{jj}
%		\end{pmatrix}\succeq0&\forall i\neq j\\
%		&y \ge 0, \ 0\le x\le 1.
%		\end{align}
%	\end{subequations}
%Relaxing the non-concave function $\max\{x_i,x_j \}$ with its concave envelope  $\min\{1, x_i+x_j\}$, we arrive at 
%\sdpfree. 
%	However, since $\max\{x_i,x_j \}$ is not a concave function of $x$, such a valid formulation is not convex. If we further replace $\max\{x_i,x_j\}$ with its concave envelope, i.e. $\min\{1, x_i+x_j\}$, the resulting formulation is valid but  further weaker  and is nothing but the convex relaxation using rank-one with free variables, i.e. \sdpfree. %We have shown \sdppos is at least as strong as \sdpfree. 
\end{proof}

The example below illustrates that \sdppos \ is indeed strictly stronger than  \optPersp=  \sdpstandard, and \sdpfree.

\begin{example}
	For $n=2$, \sdppos \  is the ideal (convex) formulation of \CQI. For the instance of \CQI with
   \[a=\begin{pmatrix}
	1\\5
	\end{pmatrix}, b=\begin{pmatrix}
	-8\\-5
	\end{pmatrix}, Q=\begin{pmatrix}
	5&2\\
	2&1\\
	\end{pmatrix}\]	
	each of the other convex relaxations has a fractional optimal solution as demonstrated in  Table~\ref{tab:comparison}.
	\begin{table}[htbp]	
		\centering 
		\caption{Comparison of convex relaxations of \CQI.}
		\begin{tabular}{l|lllll} 
			\hline \hline
				&obj val	&$x_1$&$x_2$	&$y_1$&$y_2$\\\bottomrule
%			\nat&-6.250	&0.0	&0.0					&0.0	&2.500\\
			\optPersp&-2.866&0.049&0.268&0.208&1.369\\
			\sdpstandard&-2.866&0.049&0.268&0.208&1.369\\
			\sdpfree& -2.222&0.551& 0.449&0.0&2.007\\
			\sdppos&-2.200&1.0&0.0&0.800&0.0\\\hline \hline
		\end{tabular}\label{tab:comparison}
	\end{table}
	
Notably, the fractional $x$ values for \optPersp, \sdpstandard, and 
\sdpfree \ are far from their optimal integer values. A common approach to quickly obtain feasible solutions to NP-hard problems is to round a solution obtained from a suitable convex relaxation. This example inidcates that feasible solutions obtained in this way from formulation \sdppos may be of higher quality than those obtained from weaker relaxations -- our computations in \S\ref{sec:comp_real} further corroborates this intuition.\qed
%	where $x_{nat}$ is the optimal solution provided by the natural relaxations of \eqref{eq:mixed_integer_problem} and others are optimal solutions provided by the formulations indexed by their names, correspondingly.
\end{example}

%\subsection{Convex relaxations via decomposition}
%\todo{I don't move this part as we discussed because Proposition~\ref{prop:optPersp} is required in the proof.}

An alternative way of constructing strong relaxations for \CQI is to decompose the quadratic function $y'Qy$ into a sum of univariate and bivariate convex quadratic functions and utilize the convex hull results of $2\times2$ quadratics 
\[
\alpha_{ij}q_{ij} (y_i, y_j) = \beta_{ij}x_{i}^2 \pm 2 y_i y_j + \gamma_{ij} y_{j}^2,
\]
where $\alpha_{ij} > 0$,
in Section~\ref{sec:convexHull} for each term, see \cite{frangioni2020decompositions} for such an approach. 
Specifically, let
\[
y'Qy = y'Dy + \sum_{(i,j) \in \mathcal{P}} \alpha_{ij} q_{ij}(y_i, y_j) +  
\sum_{(i,j) \in \mathcal{N}} \alpha_{ij} q_{ij}(y_i, y_j) + y'Ry
\]
where $D$ is a diagonal PSD matrix, 
$\mathcal{P}/\mathcal{N}$ is the set of quadratics $q_{ij}(\cdot)$ with positive/negative off-diagonals
 and $R$ is PSD remainder matrix. 
 %\todo{We use $D$ to define the set $d_1d_2\geq 1$, and we use to $P$ to define a general $2\times2$ matrix at the beginning of this section. Can we change the notation here or elsewhere, to avoid overloading notation? \\ Shaoning : DOne.} 
 Applying the convex hull description for each univariate and bivariate term
 we obtain the convex relaxation
\begin{equation*}\label{eq:relaxationDecomposition}
\begin{aligned}
\min\;&a'x+b'y+\sum_{i=1}^n D_{ii} y_i^2/x_i +\sum_{(i,j)\in \mathcal{P}}\alpha_{ij}f^*_+(x_i,x_j,y_i,y_j;\beta_{ij},\gamma_{ij})  \\
(\decomp) \ \ \ &+\sum_{(i,j)\in \mathcal{N}}\alpha_{ij}f^*_-(x_i,x_j,y_i,y_j;\beta_{ij},\gamma_{ij}) +y'Ry\\
\text{s.t.}\;& \ 0\le x\le 1, \ y \ge 0
\end{aligned}
\end{equation*}
for \CQI.

%\ignore{
%By utilizing Proposition~\ref{prop:extendedPos} and Remark~\ref{remark:extendedNeg}, we can obtain a convex relaxation of \eqref{eq:mixed_integer_problem} via decomposition scheme. We now show  formulation \eqref{eq:relaxationFormF} is at least as strong as the relaxation that optimally decomposes $Q$ into sums of a weakly diagonally dominant matrix and a "remainder" positive semidefinite term, and uses the perspective relaxation and functions $f_+^*$ and $f_-^*$ to strengthen the one and two-dimensional terms induced by the weakly diagonally dominant matrix.
%
%Precisely, assume the following holds for any $y\in\R^n$
%\[ y'Qy=y'Dy+\sum_{(i,j)\in \Gamma_+}c_{ij}(y_i,y_j)Q^{(ij)}(y_i,y_j)'+\sum_{(i,j)\in\Gamma_-}c_{ij}(y_i,y_j)Q^{(ij)}(y_i,y_j)' +y'Ry, \]
%where $D$ is a diagonal PSD matrix, each $c_{ij}$ is a positive number, $Q^{(ij)},(i,j)\in\Gamma_+$ is a $2\times2$ PSD matrix with the off-diagonal element as 1,  $Q^{(ij)},(i,j)\in\Gamma_-$ is a $2\times2$ PSD matrix with the off-diagonal element as -1, and $R$ is a PSD matrix. Then we have the following convex relaxation of \eqref{eq:mixed_integer_problem}.
%\begin{equation*}\label{eq:relaxationDecomposition}
%\begin{aligned}
%\min\;&a'x+b'y+\sum_{i=1}^n D_{ii} y_i^2/x_i +\sum_{(i,j)\in \Gamma_+}c_{ij}f^*_+(x_i,x_j,y_i,y_j,Q^{(ij)}_{11},Q^{(ij)}_{22})  \\
%(\decomp) \ \ \ &+\sum_{(i,j)\in \Gamma_-}c_{ij}f^*_-(x_i,x_j,y_i,y_j,Q^{(ij)}_{11},Q^{(ij)}_{22})+y'Ry\\
%\text{s.t.}\;& y \ge 0, 0\le x \le 1.
%\end{aligned}
%\end{equation*}
%}

The next proposition shows that \sdppos \  dominates \decomp for any such decomposition of the quadratic $y'Qy$.
\begin{proposition}\label{prop:relaxationViaDecomposotion}
	\sdppos \ is at least as strong as \decomp. 
\end{proposition}

\begin{proof} Let $(x,y,Y)$ be an optimal solution to \sdppos. Then $(x,y)$ is a feasible solution to \decomp.  Using the same decomposition, the objective function of \sdppos can be rewritten as 
	\begin{equation}\label{eq:decompositionObj}
	a'x+b'y+\innerProd{D}{Y}+\sum_{(i,j)\in \mathcal{P} \cup \mathcal{N}}\alpha_{ij}\innerProd{Q^{(ij)}}{Y^{(ij)}}+\innerProd{R}{Y},
	\end{equation}
	where
	$Q^{(ij)}=\begin{pmatrix}
	\beta_{ij}&\pm 1\\ \pm 1&\gamma_{ij}
	\end{pmatrix}$ and $Y^{(ij)}=\begin{pmatrix}
	Y_{ii}&Y_{ij}\\Y_{ij}&Y_{jj}
	\end{pmatrix}$.
	By Proposition~\ref{prop:optPersp}, $Y_{ii}x_i\ge y_i^2$ holds for all $i$. Hence, 
	\begin{equation}\label{eq:decompositionDiagonal}
	\innerProd{D}{Y}\ge \sum_{i=1}^nD_{ii}y_i^2/x_i.
	\end{equation}	
	By Proposition~\ref{prop:sdppos} and utilizing \eqref{eq:relaxationPos} and \eqref{eq:relaxationNeg}, we have
	\begin{subequations}\label{eq:decomposition2by2}
		\begin{align}
		\innerProd{Q^{(ij)}}{Y^{(ij)}}\ge f^*_+(x_i,x_j,y_i,y_j;\beta_{ij},\gamma_{ij})&\;\forall (i,j)\in \mathcal{P} \\
		\innerProd{Q^{(ij)}}{Y^{(ij)}}\ge f^*_-(x_i,x_j,y_i,y_j;\beta_{ij},\gamma_{ij})&\;\forall (i,j)\in \mathcal{N}.
		\end{align}
	\end{subequations}	
	Moreover, since $Y\succeq yy'$, we have
	\begin{equation}\label{eq:decompostionResidual}
	\innerProd{R}{Y}\ge \innerProd{R}{yy'}=y'Ry.
	\end{equation} 
	Combining \eqref{eq:decompositionDiagonal}, \eqref{eq:decomposition2by2}, and \eqref{eq:decompostionResidual}, the result follows.
\end{proof}

\section{Computations}\label{sec:computation}

In this section, we report on computational experiments performed to test the effectiveness the formulations derived in the paper. Section~\ref{sec:comp_synt} is devoted to synthetic portfolio optimization instances, where matrix $Q$ is diagonal dominant and the conic quadratic-representable extended formulations developed in Section~\ref{sec:convexHull} can be readily used in a branch-and-bound algorithm without the need for an SDP constraint. In Section~\ref{sec:comp_real}, we use real instances derived from stock market returns and test the SDP relaxation \sdppos derived in Section~\ref{sec:formulation}. 

\subsection{Synthetic instances -- the diagonal dominant case}\label{sec:comp_synt}

We consider a standard cardinality-constrained mean-variance portfolio optimization problem of the form
	\begin{align}
	\label{eq:portfolio}	
	\min_{x,y}\; \left \{y'Qy:
		\begin{array}{ll}
	&b'y\ge r, \ %\label{eq:portfolio_return}, \ 
	1'x\le k  \\ %\label{eq:portfolio_cardinality}\\
	&0\le y\le x, \ %\label{eq:portfolio_ub}, \ 
	x\in\{0,1\}^n %\label{eq:portfolio_int}
	\end{array}
	\right \}
	\end{align}
	\noindent
where $Q$ is the covariance matrix of returns, $b\in \R^n$ is the vector of the expected returns, 
$r$ is the target return and $k$ is the maximum number of securities in the portfolio. 
%We denote $D$ as the continuous relaxation of the feasible region of \eqref{eq:portfolio}.  
All experiments are conducted using Mosek 9.1 solver on a laptop with a 2.30GHz  $\text{Intel}^\text{\textregistered}$  $\text{Core}^{\text{\tiny TM}}$ i9-9880H CPU and 64 GB main memory. The time limit is set to one hour and all other settings are default by Mosek. 

\subsubsection{Instance generation}
We adopt the method used in \cite{atamturk2018strong} to generate the instances. The instances are designed to control the integrality gap of the instances and the effectiveness of the perspective formulation.
 Let $\rho\ge0$ be a parameter controlling the ratio of the magnitude positive off-diagonal entries of $Q$ to the magnitude of the negative off-diagonal entries of $Q$. Lower values of $\rho$ lead to higher integrality gaps.
 Let $\delta\ge0$ be the parameter controlling the diagonal dominance of $Q$. The perspective formulation is more effective in closing the integrality gap for higher values of $\delta$. The following steps are followed to generate the instances:
\begin{itemize}
	\item Construct an auxiliary matrix $\bar Q$ by drawing a factor covariance matrix $G_{20\times20}$ uniformly from $[-1,1]$, and generating an exposure matrix $H_{n\times 20}$ such that $H_{ij}=0$ with probability 0.75, and $H_{ij}$ drawn uniformly from $[0,1]$, otherwise. Let $\bar{Q}=HGG'H'$.
	\item  Construct off-diagonal entries of $Q$: For $i\neq j$, set $Q_{ij}=\bar{Q}_{ij}$, if $\bar{Q}_{ij}<0$ and set $Q_{ij}=\rho\bar Q_{ij}$ otherwise. Positive off-diagonal elements of $\bar Q$ are scaled by a factor of $\rho$.
	\item Construct diagonal entries of $Q$: Pick $\mu_i$ uniformly from $[0,\delta\bar \sigma]$, where $\bar \sigma=\frac{1}{n}\sum_{i\neq j}|Q_{ij}|$. Let $Q_{ii}=\sum_{i\neq j}|Q_{ij}|+\mu_i$. Note that if $\delta=\mu_i=0$, then matrix $Q$ is already diagonal dominant.
	\item Construct $b,r,k$: $b_i$ is drawn uniformly from $[0.5Q_{ii},1.5Q_{ii}]$,  $r=0.25\sum_{i=1}^nb_i$, and $k=\lfloor n/5 \rfloor$.
\end{itemize}  
Matrices $Q$ generated in this way have only 20.1\% of the off-diagonal entries negative on average.

\subsubsection{Formulations}
With above setting, the portfolio optimization problem can be rewritten as 
\begin{equation}\label{eq:portfolio_decomp}	
\begin{aligned}
\min&\sum_{i\in [n]}\mu_i z_i +\sum_{Q_{ij}<0}|Q_{ij}|t_{ij} +\sum_{Q_{ij}>0}|Q_{ij}|t_{ij}\\
\text{s.t.}\;\;&(x_i,y_i,z_i)\in \set X_0,\;\forall i\in N,\\
&(x_i,x_j,y_i,y_j,t_{ij})\in \set Z_-,\;\forall i\neq j: Q_{ij}<0,\\
&(x_i,x_j,y_i,y_j,t_{ij})\in \set Z_+,\;\forall i\neq j: Q_{ij}>0,\\
& b'y\ge r, \  1'x\le k, \\
%&\eqref{eq:portfolio_return}-\eqref{eq:portfolio_int},
\end{aligned}
\end{equation}
where $\set Z_+$ and $\set Z_-$ are defined as before with $d_1=d_2=1$.  
Four strong formulations are tested by replacing the mixed-integer sets with their convex hulls: \conicPersp by replacing $\set X_0$ with $\cl\conv(\set X_0)$ using the perspective reformulation
%\todo{We should not refer to this formulation as\optPersp, as we are not finding the optimal decomposition. Can you please change to name (perhaps to ConicQuadPersp?) 
%	and updates text and tables correspondingly?\\Shaoning: Done.}, 
(2) \conicN by replacing $\set X_0$ and $\set Z_-$ with $\cl\conv(\set X_0)$ and $\cl\conv(\set Z_-)$ using the corresponding extended formulation, (3) \conicP by replacing $\set X_0$ and $\set Z_+$ with $\cl\conv(\set X_0)$ and $\cl\conv(\set Z_+)$ respectively, and (4) \conic by replacing $\set X_0$, $\set Z_-$, and $\set Z_+$ with $\cl\conv(\set X_0)$, $\cl\conv(\set Z_-)$ and $\cl\conv(\set Z_+)$, correspondingly. 

\subsubsection{Results}
Table \ref{tab:delta} shows the results for matrices with varying diagonal dominance $\delta$ for $\rho=0.3$. Each row in the table represents the average for five instances generated with the same parameters. Table \ref{tab:delta} displays the dimension of the problem $n$, the initial gap (\texttt{igap}), the root gap improvement (\texttt{rimp}), the number of branch and bound nodes (\texttt{nodes}), the elapsed time in secons (\texttt{time}), and the end gap provided by the solver at termination (\texttt{egap}). In addition, in brackets, we report the number of instances solved to optimality within the time limit. The initial gap is computed as  $\texttt{igap}=\frac{\texttt{obj}_{\texttt{best}}-\texttt{obj}_{\texttt{cont}}}{|\texttt{obj}_{\texttt{best}}|}\times 100$, where
$\texttt{obj}_{\texttt{best}}$ is the objective value of the best feasible solution found and $\texttt{obj}_{\texttt{cont}}$ is the objective value of the natural continuous relaxation of \eqref{eq:portfolio}, i.e. obtained by  dropping the integral constraints; \texttt{rimp} is computed as
$\texttt{rimp}=\frac{\texttt{obj}_{\texttt{relax}}-\texttt{obj}_{\texttt{cont}}}{\texttt{obj}_{\texttt{best}}-\texttt{obj}_{\texttt{cont}}}\times 100$, where $\texttt{obj}_{\texttt{relax}}$ is the objective value of the continuous relaxation of the corresponding formulation.

In Table \ref{tab:delta}, as expected, \conicPersp has the worst performance in terms of both root gap and end gap as well as the solution time. It can only solve instances with dimension $n=40$ and some instances with dimension $n=60$ to optimality. The \texttt{rimp} of \conicPersp is less than 10\% when the diagonal dominance is small. This reflects the fact that \conicPersp provides strengthening only for diagonal terms.  \conicN  performs better than \conicPersp with \texttt{rimp} about 10\%--25\%, and it can solve all low-dimensional instances and most instances of dimension $n=60$. However, \conicN is still unable to solve high-dimensional instances effectively.  \conicP performs much better than \conicN for the instances considered: The \texttt{rimp} results in significantly stronger root improvements (between 70--80\% on average). Moreover,  \conicP can solve almost all instances to near-optimality for $n=80$. For the instances that \conicP is unable to solve to optimality, the average end gap is less than 5\%. By strengthening both the negative and positive off-diagonal terms, \conic provides the best performance with \texttt{rimp} above $90\%$. \conic can solve all instances and most of them are solved within 10 minutes.  Finally, observe that as the diagonal dominance increases, the performance of all formulations improves. Specifically, larger diagonal dominance results in more instances solved to optimality, smaller \texttt{egap} and shorter solving time for all formulations. For these instances, on average,
the gap improvement is raised from 50.69\% to 92.90\% by incorporating strengthening from off-diagonal coefficients.

Table \ref{tab:rho} displays the computational results for different values of $\rho$ with fixed $\delta=0.1$. The relative comparison of formulations is similar as discussed before, with \conic resulting in the best performance. As $\rho$ increases, the performance of \conicN deteriorates in terms of \rimp while  the performance of \conicP improves, as expected. The performance of \conic also improves for high values of $\rho$, and always results in significant improvement compared to other formulations for all instances. For these instances, on average,
the gap improvement is raised from 9.77\% to 85.38\% by incorporating strengthening from off-diagonal coefficients.

\begin{landscape}
	\begin{table}[htbp]
		\centering
		\caption{Experiments with varying diagonal dominance, $\rho=0.3$.}
		\setlength{\tabcolsep}{3pt}
		\resizebox{\linewidth}{!}{		
			\begin{tabular}{\resetRow ll^l^l^l^l^l^l^l^l^l^l^l^l^l^l^l^l^l}
				\hline
				n     & $\delta$ & \igap  &  \multicolumn{4}{l}{\conicPersp}  &        \multicolumn{4}{l}{\conicN} &  \multicolumn{4}{l}{\conicP} &  \multicolumn{4}{l}{\conic} \\
				\cmidrule(lr){4-7} \cmidrule(lr){8-11}\cmidrule(lr){12-15}\cmidrule(lr){16-19}
				&       &       
				& \rimp  & \nodes & \Time  & \egap 
				& \rimp  & \nodes & \Time  & \egap 
				& \rimp  & \nodes & \Time  & \egap  
				& \rimp  & \nodes & \Time  & \egap\\
				\bottomrule
				
				40 &       0.1 &      53.37 &       9.74 &       9,537 &         46 &    0.00[5] &      23.44 &       3,439 &         44 &    0.00[5] &      66.07 &        526 &         18 &    0.00[5] &      86.93 &         65 &         13 &    0.00[5] \\
				
				&       0.5 &      51.10 &      33.17 &       3,896 &         26 &    0.00[5] &      47.86 &       1,335 &         18 &    0.00[5] &      79.48 &        198 &          9 &    0.00[5] &      95.01 &         24 &          9 &    0.00[5] \\
				
				&       1.0 &      52.73 &      60.86 &       1,463 &          9 &    0.00[5] &      74.62 &        375 &          7 &    0.00[5] &      86.83 &        146 &          7 &    0.00[5] &      97.46 &         23 &          8 &    0.00[5] \\
				
				\multicolumn{2}{l}{\bfseries Avg} &\rowstyle{\bfseries}      52.40 &      34.59 &       4,965 &         27 &   0.00[15] &      48.64 &       1,717 &         23 &   0.00[15] &      77.46 &        290 &         11 &   0.00[15] &      93.13 &         37 &         10 &   0.00[15] \\
				
				60 &       0.1 &      46.90 &       9.05 &     316,000 &       3,363 &    5.53[1] &      19.07 &     135,052 &       3,261 &    3.83[2] &      76.78 &       4,898 &        498 &    0.00[5] &      89.72 &        445 &        140 &    0.00[5] \\
				
				&       0.5 &      50.97 &      38.46 &     134,542 &       1,888 &    2.65[3] &      49.94 &      55,434 &       1,321 &    0.98[4] &      82.72 &       1,652 &        267 &    0.00[5] &      95.13 &        203 &         75 &    0.00[5] \\
				
				&       1.0 &      47.22 &      60.04 &      21,440 &        317 &    0.00[5] &      66.52 &       8,579 &        209 &    0.00[5] &      94.69 &         86 &         35 &    0.00[5] &      98.69 &         17 &         22 &    0.00[5] \\
				
				\multicolumn{2}{l}{\bfseries Avg}  &\rowstyle{\bfseries}      48.36 &      35.85 &     157,328 &       1,856 &    2.73[9] &      45.18 &      66,355 &       1,597 &   1.60[11] &      84.73 &       2,212 &        267 &   0.00[15] &      94.51 &        222 &         79 &   0.00[15] \\
				
				80 &       0.1 &      49.91 &       4.76 &     155,000 &       3,600 &   20.25[0] &      21.96 &      69,609 &       3,600 &   14.38[0] &      65.11 &       8,017 &       2,742 &    4.69[2] &      83.33 &       2,142 &       1,416 &    0.00[5] \\
				
				&       0.5 &      50.53 &      37.33 &     136,638 &       3,600 &   12.06[0] &      49.16 &      63,897 &       3,600 &    7.49[0] &      81.57 &       6,525 &       2,473 &    1.70[2] &      94.21 &        341 &        261 &    0.00[5] \\
				
				&       1.0 &      53.78 &      56.96 &     152,704 &       3,600 &    7.41[0] &      69.41 &      45,388 &       3,068 &    2.95[2] &      84.42 &       5,870 &       2,116 &    1.27[3] &      95.67 &        365 &        275 &    0.00[5] \\
				
				\multicolumn{2}{l}{\bfseries Avg} &\rowstyle{\bfseries}      51.41 &      33.02 &     148,114 &       3,600 &   13.24[0] &      46.84 &      59,632 &       3,423 &    8.27[2] &      77.03 &       6,804 &       2,443 &    2.55[7] &      91.07 &        950 &        651 &   0.00[15] \\\hline
		\end{tabular}}%
		\label{tab:delta}%
	\end{table}

\begin{table}[htbp]
	\centering
	\caption{Experiments with varying positive off-diagonal entries, $\delta=0.1$.}
			\setlength{\tabcolsep}{3pt}
	\resizebox{\linewidth}{!}{
		\begin{tabular}{\resetRow ll^l^l^l^l^l^l^l^l^l^l^l^l^l^l^l^l^l}
			\hline
			n     & $\rho$ & \igap  &  \multicolumn{4}{l}{\conicPersp}  &        \multicolumn{4}{l}{\conicN} &  \multicolumn{4}{l}{\conicP} &  \multicolumn{4}{l}{\conic} \\
			\cmidrule(lr){4-7} \cmidrule(lr){8-11}\cmidrule(lr){12-15}\cmidrule(lr){16-19}
			&       &       
			& \rimp  & \nodes & \Time  & \egap 
			& \rimp  & \nodes & \Time  & \egap 
			& \rimp  & \nodes & \Time  & \egap  
			& \rimp  & \nodes & \Time  & \egap\\
			\bottomrule
			
			40 &        0.1 &      60.67 &       9.88 &       6,928 &         36 &    0.00[5] &      23.11 &       1,869 &         22 &    0.00[5] &      50.68 &       1,134 &         34 &    0.00[5] &      72.23 &        158 &         19 &    0.00[5] \\
			
			&        0.5 &      47.67 &        8.7 &       8,572 &         46 &    0.00[5] &      21.67 &       3,181 &         41 &    0.00[5] &      75.82 &        272 &         12 &    0.00[5] &      91.78 &         53 &         11 &    0.00[5] \\
			
			&          1.0 &      43.23 &      10.05 &       8529 &         44 &    0.00[5] &      18.08 &       4,903 &         60 &    0.00[5] &      82.33 &        149 &          8 &    0.00[5] &      92.56 &         51 &         11 &    0.00[5] \\
			
			\multicolumn{2}{l}{\bfseries Avg}  &\rowstyle{\bfseries}      50.52 &       9.54 &       8,010 &         42 &   0.00[15] &      20.95 &       3,317 &         41 &   0.00[15] &      69.61 &        519 &         18 &   0.00[15] &      85.53 &         87 &         14 &   0.00[15] \\
			
			60 &        0.1 &      60.26 &       10.7 &     256,480 &       2,585 &    4.03[2] &      27.85 &      34,563 &        847 &    0.00[5] &      53.37 &      16,190 &       1,983 &    3.10[3] &       78.5 &       1,016 &        264 &    0.00[5] \\
			
			&        0.5 &      45.98 &       9.38 &     319,534 &       3,230 &    5.57[1] &      19.21 &     103,869 &       3,043 &    4.24[2] &      78.22 &       2,715 &        315 &    0.00[5] &      91.09 &        259 &        107 &    0.00[5] \\
			
			&          1.0 &      40.87 &      10.09 &     197,140 &       3,258 &    4.52[2] &      15.93 &      98,289 &       2,982 &    4.34[2] &      85.23 &        564 &        100 &    0.00[5] &      91.66 &        135 &         72 &    0.00[5] \\
			
			\multicolumn{2}{l}{\bfseries Avg} &\rowstyle{\bfseries}      49.03 &      10.06 &     257,718 &       3,024 &    4.71[5] &         21 &      78,907 &       2,291 &    2.86[9] &      72.27 &       6,490 &        799 &   1.03[13] &      87.08 &        470 &        148 &   0.00[15] \\
			
			80 &        0.1 &      64.85 &       9.88 &     142,299 &       3,600 &   24.78[0] &      26.42 &      60,081 &       3,600 &   14.81[0] &      46.63 &      11,367 &       3,172 &   17.40[1] &       69.6 &       4,948 &       2,920 &    6.22[1] \\
			
			&        0.5 &      47.97 &       9.27 &     148,252 &       3,600 &   18.46[0] &      20.75 &      48,887 &       3,600 &   15.70[0] &       73.3 &       7,245 &       3,019 &    2.72[2] &      89.11 &       1,131 &        827 &    0.00[5] \\
			
			&          1.0 &      41.69 &         10 &     149,563 &       3,600 &   14.79[0] &      16.61 &      52,485 &       3,600 &   14.34[0] &       84.7 &       3,769 &       1,444 &    0.88[4] &      91.93 &       1,068 &        716 &    0.00[5] \\
			
			\multicolumn{2}{l}{\bfseries Avg}  &\rowstyle{\bfseries}      51.51 &       9.72 &     146,705 &       3,600 &   19.34[0] &      21.26 &      53,818 &       3,600 &   14.95[0] &      68.21 &       7,460 &       2,545 &    7.00[7] &      83.54 &       2,382 &       1,487 & 2.07[11]\\\hline
	\end{tabular}}%
	\label{tab:rho}%
\end{table}
\end{landscape}

\subsection{Real instances -- the general case}\label{sec:comp_real}

Now using real stock market data, we consider portfolio index tracking problem of the form
\begin{subequations}\label{eq:index}	
	\begin{align*}
	\min\;&(y-y_B)'Q(y-y_B)\\
\text{(IT)} \ \ \ \ 	\text{s.t. }
	&1'y=1, \ %\label{eq:index_budget}\\
	1'x\le k \\ %\label{eq:index_cardinality}\\
	&0\le y\le x, \ %\label{eq:index_ub}\\
	x\in\{0,1\}^n, %\label{eq:index_int}
	\end{align*}
\end{subequations}
where $y_B\in \R^n$ is a benchmark index portfolio, $Q$ is the covariance matrix of security returns and 
$k$ is the maximum number of securities in the portfolio. 
%We denote $D$ as the continuous relaxation of the feasible region of \eqref{eq:portfolio}.  

\subsubsection{Instance generation} We use the daily stock return data provided by Boris Marjanovic in Kaggle\footnote{https://www.kaggle.com/borismarjanovic/price-volume-data-for-all-us-stocks-etfs} to compute the covariance matrix $Q$. Specifically, given a desired start date (either 1/1/2010 or 1/1/2015 in our computations), we compute the sample covariance matrix based on the stocks with available data in at least 99\% of the days since the start (returns for missing data are set to 0). The resulting covariance matrices are available at \url{https://sites.google.com/usc.edu/gomez/data}. We then generate instances as follows:
\begin{itemize}
	\item we randomly sample an $n\times n$ covariance matrix $Q$ corresponding to $n$ stocks, and
	\item we draw each element of $y_B$ from uniform [0,1], and then scale $y_B$ so that $1'y_B=1$. 
\end{itemize}

\subsubsection{Formulations}
%We consider convex relaxations of \eqref{eq:index}. First, we point out that 
The natural convex relaxation of IT always yields a trivial lower bound of $0$, as it is possible to set $z=y=y_B$. Thus, we do not report results concerning the natural relaxation. Instead, we consider the optimal perspective relaxation \optPersp of
 \cite{dong2015regularization}:
\begin{align}
\min_{x,y,Y}\;& \zeta'Q\zeta-2\zeta'Qy+\innerProd{Q}{Y}\\
\text{s.t.}\;&Y-yy'\succeq 0\\
(\perspS)\qquad\qquad&y_i^2\leq Y_{ii}x_i&\forall i\in[n]\\
&0\le x \le 1, \ y \ge 0 \\
	&1'y=1, \ %\label{eq:index_budget}\\
1'x\le k 
%&\eqref{eq:index_budget}-\eqref{eq:index_ub},
\end{align}
and the proposed \sdppos exploiting off-diagonal elements of $Q$:
\small\begin{align*}
\min_{x,y,Y,W}\;& \zeta'Q\zeta-2\zeta'Qy+\innerProd{Q}{Y}\\
\text{s.t.}\;&Y-yy'\succeq 0\\
& W^{ij}\succeq 0 &\forall i<j\\
&(Y_{ii}-W_{11}^{ij})(x_i-W_{33}^{ij})\ge(y_i-W_{31}^{ij})^2,\;W_{11}^{ij}\le Y_{ii}&\forall i<j\\ %\label{eq:sdpposi}\\
(\conicS)\qquad\qquad&(Y_{jj}-W_{22}^{ij})(x_j-W_{33}^{ij})\ge(y_j-W_{32}^{ij})^2,\;W_{22}^{ij}\le Y_{jj}&\forall i<j \\ %&\forall i< j\label{eq:sdpposj}\\
& W_{33}^{ij}\leq x_i+x_j-1,\;W_{33}^{ij}\le x_i,\;W_{33}^{ij}\le x_j&\forall i<j\\
&0\leq W_{31}^{ij}\leq y_i,\; 0\leq W_{32}^{ij}\leq y_j,\;W_{12}^{ij}=Y_{ij}&\forall i<j\\
&0\le x \le 1, \ y \ge 0 \\
	&1'y=1, \ %\label{eq:index_budget}\\
1'x\le k,
%&\eqref{eq:index_budget}-\eqref{eq:index_ub}.
\end{align*}\normalsize

In addition, for each relaxation, we consider a simple rounding heuristic to obtain feasible solutions to (IT): given an optimal solution $(\bar x,\bar y)$ to the continuous relaxation, we fix $x_i=1$ for the $k$-largest values of $\bar x$ and the remaining $x_i=0$, and resolve the continuous relaxation to compute $y$.  

\subsubsection{Results} Tables~\ref{tab:real2010} and \ref{tab:real2015} present the results using historical data since 2010 and 2015, respectively. They show, for different values of $n$ and $k$, and for each conic relaxation, the time required to solve the instances in seconds, the lower bound (\texttt{LB}) corresponding to the optimal objective value of the continuous relaxation, the upper bound (\texttt{UB}) corresponding to the objective value of the heuristic, and the gap between these two values computed as $\texttt{Gap}=\frac{\texttt{UB}-\texttt{LB}}{\texttt{UB}}$. The lower and upper bounds are scaled so that the best upper bound found for a given instance is $100$. Bound and gap values represent an average of five instances generated with the same parameters, while times are averages of 15 instances across all cardinalities. In addition, Figure~\ref{fig:gaps} reports the distribution of gaps across all instances.

\begin{table}[h]
	\centering
	\caption{Results with stock return data since 2010.}
	%, resulting in 1,979 time epochs to compute the sample covariance matrix. On average, 98\% of the off-diagonal elements of the covariance matrices are nonnegative.}
	\label{tab:real2010}
	\setlength{\tabcolsep}{3pt}
	\scalebox{0.9}{
	\begin{tabular}{c  l c | c c c | c c c |c c c}
		\hline \hline
		\multirow{2}{*}{$n$}&\multirow{2}{*}{\texttt{Method}}&\multirow{2}{*}{\texttt{Time(s)}}& \multicolumn{3}{c|}{$k=.10n$}&\multicolumn{3}{c|}{$k=.15n$}&\multicolumn{3}{c}{$k=.20n$}\\
		&&&\texttt{LB} & \texttt{UB} & \texttt{Gap}&\texttt{LB} & \texttt{UB} & \texttt{Gap}&\texttt{LB} & \texttt{UB} & \texttt{Gap}\\
		\hline
		\multirow{2}{*}{50}&\perspS&1.1&93.6 & 105.3 & 10.9\% & 91.9 & 104.5 & 12.0\% & 89.9 & 103.5 & 13.2\%\\
		&\conicS&3.3&98.6 & 100.0 & 1.4\% & 97.4 & 100.0 & 2.6\% & 95.1 & 102.2 & 6.8\%\\	
		\hline
		&&&&&&&&&&&\\
		\multirow{2}{*}{100}&\perspS&30.8&90.7 & 106.9 & 14.6\% & 90.8 & 108.1 & 15.4\% & 90.5 & 105.9 & 14.5\%\\
		&\conicS&65.7&98.5 & 100.1 & 1.6\% & 99.2 & 100.0 & 0.8\% & 99.0 & 100.0 & 1.0\%\\
		\hline		
		&&&&&&&&&\\
		\multirow{2}{*}{150}&\perspS&230.5&80.7 & 116.8 & 30.3\% & 81.6 & 116.0 & 28.8\% & 78.8 & 119.3 & 32.5\%\\
		&\conicS&465.4&92.6 & 100.3 & 7.7\% & 93.5 & 100.0 & 6.5\% & 90.1 & 100.0 & 9.9\%\\
		\hline \hline
	\end{tabular}
	}
\end{table}

\begin{table}[h]
	\centering
	\caption{Results with stock return data since 2015.}
	%, resulting in 721 time epochs to compute the sample covariance matrix. On average, 93\% of the off-diagonal elements of the covariance matrices are nonnegative.}
	\label{tab:real2015}
	\setlength{\tabcolsep}{3pt}
	\scalebox{0.9}{
	\begin{tabular}{c  l c | c c c | c c c |c c c}
		\hline \hline
		\multirow{2}{*}{$n$}&\multirow{2}{*}{\texttt{Method}}&\multirow{2}{*}{\texttt{Time(s)}}& \multicolumn{3}{c|}{$k=.10n$}&\multicolumn{3}{c|}{$k=.15n$}&\multicolumn{3}{c}{$k=.20n$}\\
		&&&\texttt{LB} & \texttt{UB} & \texttt{Gap}&\texttt{LB} & \texttt{UB} & \texttt{Gap}&\texttt{LB} & \texttt{UB} & \texttt{Gap}\\
		\hline
		\multirow{2}{*}{50}&\perspS&1.2&90.6 & 116.3 & 20.9\% & 91.4 & 116.9 & 20.9\% & 90.3 & 114.1 & 20.2\%\\
		&\conicS&3.5&99.3 & 100.0 & 0.7\% & 99.5 & 100.0 & 0.5\% & 98.9 & 100.0 & 1.1\%\\
		\hline
		&&&&&&&&&&&\\
		\multirow{2}{*}{100}&\perspS&23.6&78.2 & 240.3 & 56.5\% & 76.1 & 264.8 & 66.2\% & 75.9 & 188.2 & 49.6\%\\
		&\conicS&51.3&94.0 & 100.0 & 6.0\% & 91.5 & 100.0 & 8.5\% & 90.6 & 100.0 & 9.4\%\\
		\hline		
		&&&&&&&&&\\
		\multirow{2}{*}{150}&\perspS&177.5&51.8 & 178.2 & 68.0\% & 46.4 & 223.4 & 77.5\% & 48.4 & 155.6 & 64.2\%\\
		&\conicS&352.3&66.2 & 100.0 & 33.8\% & 59.0 & 100.0 & 41.0\% & 61.2 & 100.0 & 38.8\%\\
		\hline \hline
	\end{tabular}
	}
\end{table}

\begin{figure}
	\subfloat[Data since 2010.]{\includegraphics[width=0.5\textwidth,trim={10cm 4cm 10cm 4cm},clip]{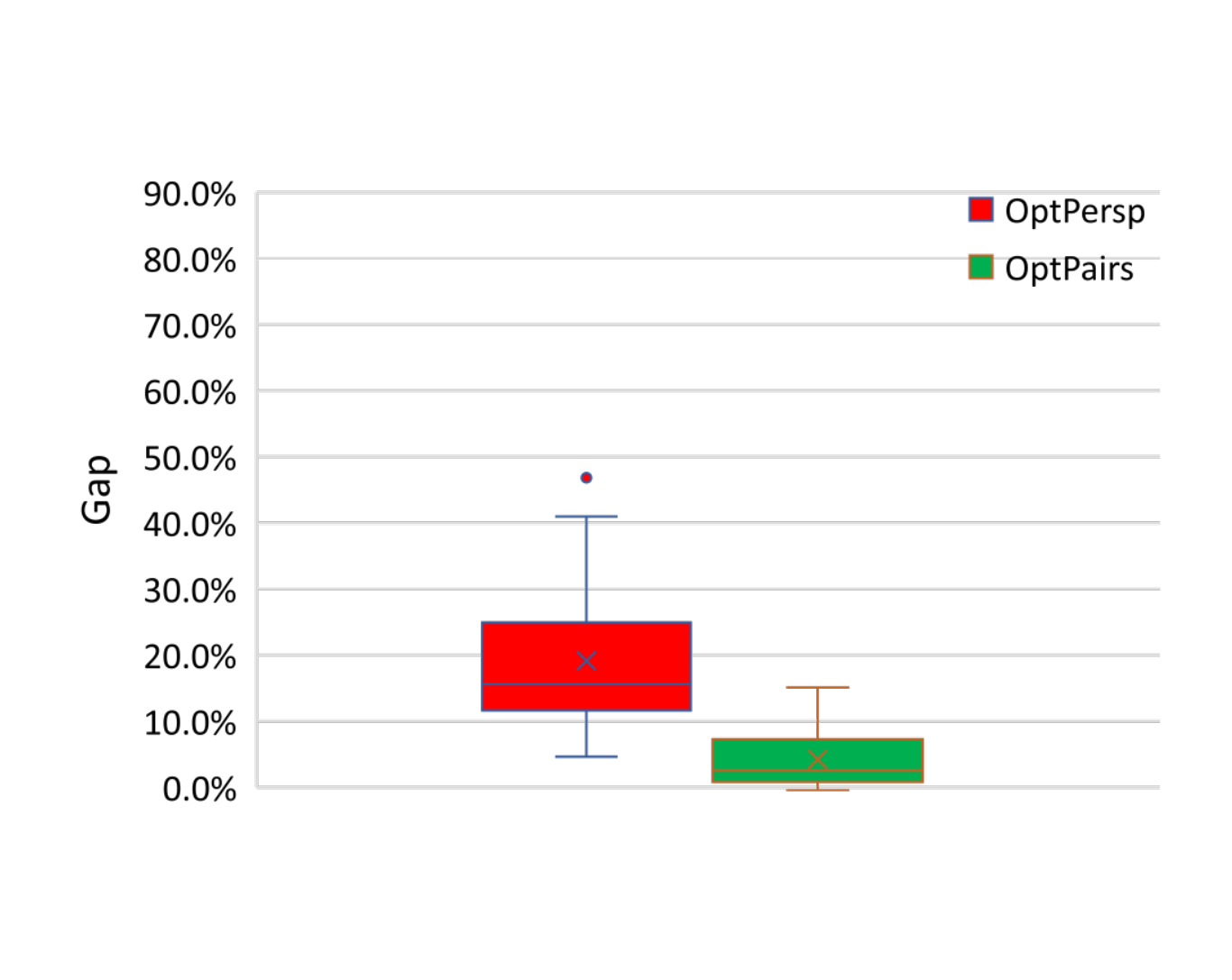}}\hfill
	\subfloat[Data since 2015.]{\includegraphics[width=0.5\textwidth,trim={10cm 4cm 10cm 4cm},clip]{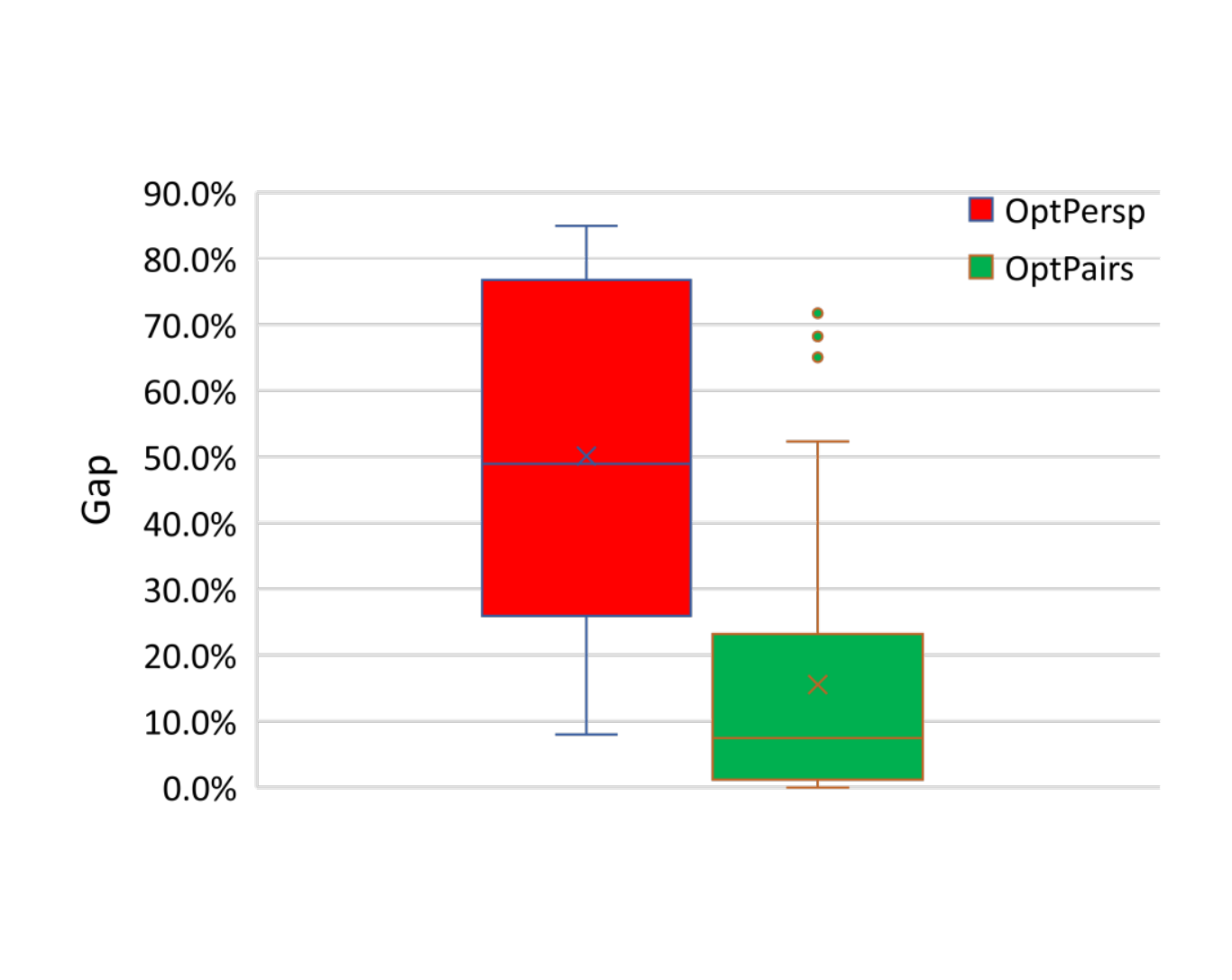}}
	\caption{Distribution of gaps for \perspS\ and \conicS.}
	\label{fig:gaps}
\end{figure}

Observe that \conicS\ consistently delivers higher quality solutions than \perspS\ across all values of $n$ and $k$, both in terms of lower and upper bounds, and leads to significant reduction of the gaps: for data since 2010, \perspS\ yields an average gap of 19.1\%, whereas \conicS\ yields an average gap of 4.2\%; for data since 2015, \perspS\ yields an average gap of 50.1\%, whereas \conicS\ yields an average gap of 15.5\%. The upper bounds obtained from rounding the solution of \perspS\ appear to be especially poor for instances with data since 2015, resulting in values two times larger than those obtained from \conicS. The improved upper bounds suggest that, in addition to delivering improved lower bounds, the feasible solutions obtained from \conicS\ may be closer to optimal solutions. Nonetheless, given the simplicity of the rounding heuristic, we expect that better upper bounds can be found (and lower bounds obtained from either relaxation are closer to the true optimal value than indicated here). 

With respect to solution times, we report the following encouraging results. First, while \conicS requires about two to three times more time than \perspS, this factor does not seem to be affected by the dimension of the problem $n$, and thus both methods scale similarly. Second, the solution times reported here are orders-of-magnitude smaller than those reported in \cite{frangioni2020decompositions} using disjunctive extended formulations to solve decomposition problems (though instances and the computational setup are different): in \cite{frangioni2020decompositions}, the authors report times in the order of $10^5$ seconds to solve optimal decompositions in instances with $n=50$; in contrast, problems with similar size are solved in under four seconds here, % (five orders-of-magnitude improvement), 
and problems with $n=150$ can be solved within minutes to optimality. These results clearly illustrate the benefits of deriving ideal formulations in the original space of variables or tight extended formulations.
% instead of relying on large extended formulations obtained from disjunctive programming. 

\section{Conclusions}\label{sec:conclusions}
In this paper, we first show the equivalence between two well-known convex relaxations -- Shor's SDP and optimal perspective formulation for \CQI. Then we describe the convex hull of the mixed-integer epigraph of the bivariate convex quadratic functions with nonnegative variables and off-diagonals with an SOCP-representable extended formulation as well as in the original space of variables. Furthermore, we develop a new technique for constructing an optimal convex relaxation from elementary valid inequalities. Using this technique, we develop a new strong SDP relaxation for \CQI, based on the convex hull descriptions of the bivariate cases as building blocks. Moreover,
the computational results with synthetic and real portfolio optimization instances indicate that the proposed formulations provide substantial improvement over existing alternatives in the literature.

%In addition to good theoretical properties, the computational results show the proposed formulation is very powerful and general enough to permit additional constraints. When the matrix is diagonal dominant, the derived formulation has a significant root improvement over 90\% and is able to solve most instances with $n=80$ to optimality in about 10 minutes, while the optimal perspective reformulation fail to  solve any of them within one hour. The computational experiments using real stock data also indicate that the novel SDP formulation can close the optimality gap by orders of magnitude compared to the existing methods.

\section*{Acknowledgments}
Andr\'es G\'omez is supported, in part, by grants 1930582 and 1818700 from the National Science Foundation. Alper Atamt\"urk is supported, in part, NSF grant 1807260, DOE ARPA-E grant 260801540061, and DOD ONR grant 12951270.

\linespread{1.0}

\bibliographystyle{apalike}
\bibliography{Bibliography}

\begin{thebibliography}{}

\bibitem[Akt{\"u}rk et~al., 2009]{akturk2009strong}
Akt{\"u}rk, M.~S., Atamt{\"u}rk, A., and G{\"u}rel, S. (2009).
\newblock A strong conic quadratic reformulation for machine-job assignment
  with controllable processing times.
\newblock {\em Operations Research Letters}, 37(3):187--191.

\bibitem[Alfakih et~al., 1999]{alfakih1999solving}
Alfakih, A.~Y., Khandani, A., and Wolkowicz, H. (1999).
\newblock Solving euclidean distance matrix completion problems via
  semidefinite programming.
\newblock {\em Computational Optimization and Applications}, 12(1-3):13--30.

\bibitem[Anstreicher and Burer, 2010]{lowdim-quad}
Anstreicher, K. and Burer, S. (2010).
\newblock Computable representations for convex hulls of low-dimensional
  quadratic forms.
\newblock {\em Mathematical Programming}, 124(1):33--43.

\bibitem[Atamt{\"u}rk and G{\'o}mez, 2018]{atamturk2018strong}
Atamt{\"u}rk, A. and G{\'o}mez, A. (2018).
\newblock Strong formulations for quadratic optimization with {M}-matrices and
  indicator variables.
\newblock {\em Mathematical Programming}, 170(1):141--176.

\bibitem[Atamt\"urk and G\'omez, 2019]{atamturk2019rank}
Atamt\"urk, A. and G\'omez, A. (2019).
\newblock Rank-one convexification for sparse regression.
\newblock {\em arXiv preprint arXiv:1901.10334}.

\bibitem[Atamt\"urk and G\'omez, 2020]{atamturk2020safe}
Atamt\"urk, A. and G\'omez, A. (2020).
\newblock Safe screening rules for {$\ell_0$}-regression.
\newblock {\em
  http://www.optimization-online.org\//DB\_HTML/2020/02/7642.html}.

\bibitem[Atamt{\"u}rk et~al., 2018]{atamturk2018signal}
Atamt{\"u}rk, A., G{\'o}mez, A., and Han, S. (2018).
\newblock Sparse and smooth signal estimation: Convexification of
  {$\ell_0$}-formulations.
\newblock {\em arXiv preprint arXiv:1811.02655}.

\bibitem[Bach, 2019]{bach2019submodular}
Bach, F. (2019).
\newblock Submodular functions: from discrete to continuous domains.
\newblock {\em Mathematical Programming}, 175(1-2):419--459.

\bibitem[Belotti et~al., 2016]{belotti2016handling}
Belotti, P., Bonami, P., Fischetti, M., Lodi, A., Monaci, M.,
  Nogales-G{\'o}mez, A., and Salvagnin, D. (2016).
\newblock On handling indicator constraints in mixed integer programming.
\newblock {\em Computational Optimization and Applications}, 65(3):545--566.

\bibitem[Ben-Tal et~al., 2009]{ben2009robust}
Ben-Tal, A., El~Ghaoui, L., and Nemirovski, A. (2009).
\newblock {\em Robust Optimization}, volume~28.
\newblock Princeton University Press.

\bibitem[Bertsimas et~al., 2019]{bertsimas2019unified}
Bertsimas, D., Cory-Wright, R., and Pauphilet, J. (2019).
\newblock A unified approach to mixed-integer optimization: Nonlinear
  formulations and scalable algorithms.
\newblock {\em arXiv preprint arXiv:1907.02109}.

\bibitem[Bienstock, 1996]{bienstock1996computational}
Bienstock, D. (1996).
\newblock Computational study of a family of mixed-integer quadratic
  programming problems.
\newblock {\em Mathematical Programming}, 74(2):121--140.

\bibitem[Bienstock and Michalka, 2014]{bienstock2014cutting}
Bienstock, D. and Michalka, A. (2014).
\newblock Cutting-planes for optimization of convex functions over nonconvex
  sets.
\newblock {\em SIAM Journal on Optimization}, 24(2):643--677.

\bibitem[Boland et~al., 2017]{boland2017bounding}
Boland, N., Dey, S.~S., Kalinowski, T., Molinaro, M., and Rigterink, F. (2017).
\newblock Bounding the gap between the {McCormick} relaxation and the convex
  hull for bilinear functions.
\newblock {\em Mathematical Programming}, 162(1):523--535.

\bibitem[Boman et~al., 2005]{boman2005factor}
Boman, E.~G., Chen, D., Parekh, O., and Toledo, S. (2005).
\newblock On factor width and symmetric {H}-matrices.
\newblock {\em Linear Algebra and Its Applications}, 405:239--248.

\bibitem[Bonami et~al., 2015]{bonami2015mathematical}
Bonami, P., Lodi, A., Tramontani, A., and Wiese, S. (2015).
\newblock On mathematical programming with indicator constraints.
\newblock {\em Mathematical Programming}, 151(1):191--223.

\bibitem[Burer and Anstreicher, 2020]{burer2020quadratic}
Burer, S. and Anstreicher, K. (2020).
\newblock Quadratic optimization with switching variables: The convex hull for
  $ n= 2$.
\newblock {\em arXiv preprint arXiv:2002.04681}.

\bibitem[Burer and Ye, 2019]{burer2019exact}
Burer, S. and Ye, Y. (2019).
\newblock Exact semidefinite formulations for a class of (random and
  non-random) nonconvex quadratic programs.
\newblock {\em Mathematical Programming}, pages 1--17.
\newblock https://doi.org/10.1007/s10107-019-01367-2.

\bibitem[{Candes} and {Plan}, 2010]{candes2010matrix}
{Candes}, E.~J. and {Plan}, Y. (2010).
\newblock Matrix completion with noise.
\newblock {\em Proceedings of the IEEE}, 98(6):925--936.

\bibitem[Dedieu et~al., 2020]{dedieu2020learning}
Dedieu, A., Hazimeh, H., and Mazumder, R. (2020).
\newblock Learning sparse classifiers: Continuous and mixed integer
  optimization perspectives.
\newblock {\em arXiv preprint arXiv:2001.06471}.

\bibitem[Dey et~al., 2019]{dey2019new}
Dey, S.~S., Santana, A., and Wang, Y. (2019).
\newblock New socp relaxation and branching rule for bipartite bilinear
  programs.
\newblock {\em Optimization and Engineering}, 20(2):307--336.

\bibitem[Dong et~al., 2015]{dong2015regularization}
Dong, H., Chen, K., and Linderoth, J. (2015).
\newblock Regularization vs. relaxation: A conic optimization perspective of
  statistical variable selection.
\newblock {\em arXiv preprint arXiv:1510.06083}.

\bibitem[Dong and Linderoth, 2013]{dong2013}
Dong, H. and Linderoth, J. (2013).
\newblock On valid inequalities for quadratic programming with continuous
  variables and binary indicators.
\newblock In Goemans, M. and Correa, J., editors, {\em Proceedings of IPCO
  2013}, page 169–180, Berlin. Springer.

\bibitem[Fattahi et~al., 2017]{FALA:conic-uc}
Fattahi, S., Ashraphijuo, M., Lavaei, J., and Atamt{\"u}rk, A. (2017).
\newblock Conic relaxations of the unit commitment problem.
\newblock {\em Energy}, 134:1079--1095.

\bibitem[Frangioni and Gentile, 2006]{frangioni2006perspective}
Frangioni, A. and Gentile, C. (2006).
\newblock Perspective cuts for a class of convex 0--1 mixed integer programs.
\newblock {\em Mathematical Programming}, 106(2):225--236.

\bibitem[Frangioni et~al., 2020]{frangioni2020decompositions}
Frangioni, A., Gentile, C., and Hungerford, J. (2020).
\newblock Decompositions of semidefinite matrices and the perspective
  reformulation of nonseparable quadratic programs.
\newblock {\em Mathematics of Operations Research}, 45(1):15--33.

\bibitem[Gao and Li, 2011]{gao2011cardinality}
Gao, J. and Li, D. (2011).
\newblock Cardinality constrained linear-quadratic optimal control.
\newblock {\em IEEE Transactions on Automatic Control}, 56(8):1936--1941.

\bibitem[Goemans and Williamson, 1995]{goemans1995improved}
Goemans, M.~X. and Williamson, D.~P. (1995).
\newblock Improved approximation algorithms for maximum cut and satisfiability
  problems using semidefinite programming.
\newblock {\em Journal of the ACM (JACM)}, 42(6):1115--1145.

\bibitem[G\'omez, 2019]{gomez2019outlier}
G\'omez, A. (2019).
\newblock Outlier detection in time series via mixed-integer conic quadratic
  optimization.
\newblock {\em http://www.optimization-online.org/DB\_HTML/2019/11/7488.html}.

\bibitem[G\'omez, 2020]{gomez2020strong}
G\'omez, A. (2020).
\newblock Strong formulations for conic quadratic optimization with indicator
  variables.
\newblock {\em Forthcoming in Mathematical Programming}.

\bibitem[G{\"u}nl{\"u}k and Linderoth, 2010]{gunluk2010perspective}
G{\"u}nl{\"u}k, O. and Linderoth, J. (2010).
\newblock Perspective reformulations of mixed integer nonlinear programs with
  indicator variables.
\newblock {\em Mathematical Programming}, 124:183--205.

\bibitem[Gupte et~al., 2020]{gupte2020extended}
Gupte, A., Kalinowski, T., Rigterink, F., and Waterer, H. (2020).
\newblock Extended formulations for convex hulls of some bilinear functions.
\newblock {\em Discrete Optimization}, 36:100569.

\bibitem[Hijazi et~al., 2012]{hijazi2012mixed}
Hijazi, H., Bonami, P., Cornu{\'e}jols, G., and Ouorou, A. (2012).
\newblock Mixed-integer nonlinear programs featuring “on/off” constraints.
\newblock {\em Computational Optimization and Applications}, 52:537--558.

\bibitem[Ho-Nguyen and K{\i}l{\i}n\c{c}-Karzan, 2017]{ho2017second}
Ho-Nguyen, N. and K{\i}l{\i}n\c{c}-Karzan, F. (2017).
\newblock A second-order cone based approach for solving the trust-region
  subproblem and its variants.
\newblock {\em SIAM Journal on Optimization}, 27(3):1485--1512.

\bibitem[Hochbaum, 2001]{hochbaum2001efficient}
Hochbaum, D.~S. (2001).
\newblock An efficient algorithm for image segmentation, {M}arkov random fields
  and related problems.
\newblock {\em Journal of the {ACM}}, 48:686--701.

\bibitem[Javanmard et~al., 2016]{javanmard2016phase}
Javanmard, A., Montanari, A., and Ricci-Tersenghi, F. (2016).
\newblock Phase transitions in semidefinite relaxations.
\newblock {\em Proceedings of the National Academy of Sciences},
  113(16):E2218--E2223.

\bibitem[Jeon et~al., 2017]{jeon2017quadratic}
Jeon, H., Linderoth, J., and Miller, A. (2017).
\newblock Quadratic cone cutting surfaces for quadratic programs with on--off
  constraints.
\newblock {\em Discrete Optimization}, 24:32--50.

\bibitem[Jeyakumar and Li, 2014]{jeyakumar2014trust}
Jeyakumar, V. and Li, G. (2014).
\newblock Trust-region problems with linear inequality constraints: exact {SDP}
  relaxation, global optimality and robust optimization.
\newblock {\em Mathematical Programming}, 147(1-2):171--206.

\bibitem[Lim et~al., 2018]{lim2018valid}
Lim, C.~H., Linderoth, J., and Luedtke, J. (2018).
\newblock Valid inequalities for separable concave constraints with indicator
  variables.
\newblock {\em Mathematical Programming}, 172(1-2):415--442.

\bibitem[Locatelli and Schoen, 2014]{conv-bivariate}
Locatelli, M. and Schoen, F. (2014).
\newblock On convex envelopes for bivariate functions over polytopes.
\newblock {\em Mathematical Programming}, 144(1):56--91.

\bibitem[Mahajan et~al., 2017]{mahajan2017minotaur}
Mahajan, A., Leyffer, S., Linderoth, J., Luedtke, J., and Munson, T. (2017).
\newblock Minotaur: A mixed-integer nonlinear optimization toolkit.
\newblock Technical report, ANL/MCS-P8010-0817, Argonne National Lab.

\bibitem[Shor, 1987]{shor1987quadratic}
Shor, N.~Z. (1987).
\newblock Quadratic optimization problems.
\newblock {\em Soviet Journal of Computer and Systems Sciences}, 25:1--11.

\bibitem[Wang and K{\i}l{\i}n{\c{c}}-Karzan, 2019]{wang2019tightness}
Wang, A.~L. and K{\i}l{\i}n{\c{c}}-Karzan, F. (2019).
\newblock On the tightness of {SDP} relaxations of {QCQP}s.
\newblock {\em http://www.optimization-online.org/DB\_HTML/2019/11/7487.html}.

\bibitem[Wang and K{\i}l{\i}n\c{c}-Karzan, 2019]{wang2019generalized}
Wang, A.~L. and K{\i}l{\i}n\c{c}-Karzan, F. (2019).
\newblock The generalized trust region subproblem: solution complexity and
  convex hull results.
\newblock {\em arXiv preprint arXiv:1907.08843}.

\bibitem[Wu et~al., 2017]{wu2017quadratic}
Wu, B., Sun, X., Li, D., and Zheng, X. (2017).
\newblock Quadratic convex reformulations for semicontinuous quadratic
  programming.
\newblock {\em {SIAM} Journal on Optimization}, 27:1531--1553.

\end{thebibliography}

\end{document}